\documentclass[a4paper,12pt,oneside]{amsart}

\usepackage{fouriernc}
\usepackage{amssymb}  
\usepackage{latexsym} 
\usepackage{comment}
\usepackage{color}
\usepackage{colonequals}
\usepackage{makecell}

\usepackage{mathtools}
\mathtoolsset{showonlyrefs=true}

\usepackage{amsmath,amsthm}
\usepackage[all]{xy}

 \usepackage{epsfig}
\usepackage{tikz}
\usepackage{tikz-cd}

\definecolor{darkgreen}{rgb}{0,0.5,0}
\usepackage[
        colorlinks, citecolor=darkgreen,
         backref,       
        pdfauthor={Federico Fallucca, Roberto Pignatelli, Francesco Polizzi},
        pdftitle={The Picard number of  fibred Mori dream surfaces},
        linktocpage   
]{hyperref}

\usepackage[lite]{amsrefs} 

\usepackage[all]{xy} 

\usepackage{enumerate} 

\renewcommand{\to}{\longrightarrow}

\newcommand{\CC}{\ensuremath{\mathbb{C}}}
\newcommand{\FF}{\ensuremath{\mathbb{F}}}

\newcommand{\PP}{\ensuremath{\mathbb{P}}}

\newcommand{\RR}{\ensuremath{\mathbb{R}}}

\newcommand{\lr}{\longrightarrow}

\DeclareMathOperator{\GL}{GL}

\DeclareMathOperator{\Pic}{Pic}
\DeclareMathOperator{\PSL}{PSL}
\DeclareMathOperator{\Sing}{Sing}

\def\Bbb{\bf}

\def\psEff{\overline{\text{Eff}}}

\newcommand\dual{\mathrel{\raise3pt\hbox{$\underline{\mathrm{\thinspace d
\thinspace}}$}}}
\newcommand\qe{\ifhmode\unskip\nobreak\fi\quad $\Box$}       

\def\BOX{\hfill\lower.5\baselineskip\hbox{$\Box$}}

\makeatletter
\newtheorem{innercustomthm}{Theorem}
\newenvironment{customthm}[1]
  {\renewcommand\theinnercustomthm{#1}%
   \protected@edef\@currentlabel{#1}%
   \innercustomthm}
  {\endinnercustomthm}
 \makeatother

\makeatletter
\newtheorem{innercustomcor}{Corollary}
\newenvironment{customcor}[1]
  {\renewcommand\theinnercustomcor{#1}%
   \protected@edef\@currentlabel{#1}%
   \innercustomcor}
  {\endinnercustomcor}
 \makeatother

\newtheorem{theorem}{Theorem}[section]
\newtheorem{lemma}[theorem]{Lemma}
\newtheorem{corollary}[theorem]{Corollary}
\newtheorem{proposition}[theorem]{Proposition}

\newtheorem*{theorem*}{Theorem}
\newtheorem*{problem*}{Problem}
\newtheorem*{question*}{Question}

\theoremstyle{remark}

\theoremstyle{definition}
\newtheorem{definition}[theorem]{Definition}

\newtheorem{remark}[theorem]{Remark}
\newtheorem{exampleJS}[theorem]{Example (J. Starr)}
\newtheorem{exampleWS}[theorem]{Example (W. Sawin)}

\numberwithin{equation}{section}

\newcounter{nootje}
\renewcommand\check[1]
  {\marginpar{\tiny\begin{minipage}{20mm}\begin{flushleft}\thenootje : #1\end{flushleft}\end{minipage}}\addtocounter{nootje}{1}}
\begin{document}

\title[The Picard number of fibred Mori dream surfaces]{The Picard number of fibred Mori dream surfaces}

\author{Federico Fallucca}
\address{Dipartimento di Matematica,
	Universit\`a di Trento,
	via Sommarive 14,
	I-38123 Trento, Italy.}
\email{federico.fallucca@unitn.it}
\email{fallucca@altamatematica.it}
\author{Roberto Pignatelli}
\address{Dipartimento di Matematica,
	Universit\`a di Trento,
	via Sommarive 14,
	I-38123 Trento, Italy.}
\email{Roberto.Pignatelli@unitn.it}
\author{Francesco Polizzi}
\address{Dipartimento di Matematica e Applicazioni "Renato Cacioppoli", Universit\`a degli Studi di Napoli Federico II, Via Cintia, Monte S.Angelo, I-80126 Napoli, Italy.}
\email{francesco.polizzi@unina.it}
\date{\today}

 \makeatletter
    \def\@pnumwidth{2em}
  \makeatother

\begin{abstract}
Let $S$ be a smooth projective surface endowed with a fibration
$f \colon S \to C$ onto a smooth projective curve. We prove that, if the
pseudo-effective cone $\overline{\operatorname{Eff}}(S)$ is polyhedral, then
the Picard number $\rho(S)$ is determined by the number of irreducible
components of the reducible fibres of $f$. 

In the special case where $S$ is a Mori dream space, our argument provides a new proof,
based on different methods, of the corresponding equality, which also follows
from a result of Casagrande.

As a further consequence, if $f \colon S \to \mathbb{P}^1$ is an elliptic
fibration with a section and $S$ is a Mori dream space, then the
Mordell--Weil group of its generic fibre is finite.

The main results of the paper are applications of this non-polyhedrality
criterion. We prove that certain Horikawa surfaces, Fermat surfaces in
$\mathbb{P}^3$ of every degree $d \geq 4$, some product-quotient surfaces,
and the minimal simply connected numerical Godeaux surface constructed by
Craighero and Gattazzo have non-polyhedral pseudo-effective cones. In
particular, none of these surfaces is a Mori dream space.
\end{abstract}

\maketitle

\vspace{-0.5cm}

\tableofcontents

\section{Introduction}  \label{sec:intro}

Since the pioneering work of Hu and Keel \cite{HuKeel00}, Mori dream
spaces have been intensively studied. In particular, considerable
attention has been devoted to determining which algebraic varieties
are Mori dream spaces and which are not.

Research on this topic has so far concerned mainly, though not
exclusively, rational and Fano varieties. By contrast, varieties of
general type remain much less understood, and comparatively few
examples have been investigated; see, for instance,
\cite{CCCK,FrapLee22,KL19}. The present work contributes to this
direction.
 
The key result underlying our applications is the following, see Theorem \ref{thm: criterion}:

\begin{customthm}{A} \label{thm: main}
Let $f \colon S \to C$ be a fibration from a smooth complex projective
surface $S$ onto a smooth projective curve $C$. If the pseudo-effective
cone $\overline{\operatorname{Eff}}(S)$ is polyhedral, then
\begin{equation}
\rho(S) = n(f)+2.
\end{equation}
\end{customthm}

Here $\rho(S)$ denotes the Picard number of $S$. The integer $n(f)$ is
a weighted count of the reducible fibres: a fibre with exactly $k$
distinct irreducible components contributes $k-1$ to $n(f)$, see Definition \ref{def: N}.

When $S$ is a Mori dream space, the equality in Theorem \ref{thm: main}
is also a special case of a result by Casagrande 
\cite{Casagrande}*{Corollary 2.16}, applied to the fibration
$f \colon S \to C$. Indeed, such a fibration is a special contraction,
$\rho(C)=1$, the image in $N_1(S)$ of the numerical group of a general
fibre has dimension one, and the remaining terms in Casagrande's
formula are precisely the contributions $k_p(f)-1$ of the reducible
fibres. Our proof uses different methods and, more importantly,
requires only the weaker assumption that
$\overline{\operatorname{Eff}}(S)$ is polyhedral.

In particular, Theorem \ref{thm: main} implies that, if
$\overline{\operatorname{Eff}}(S)$ is polyhedral, then all fibrations
on $S$ have the same weighted number $n(f)$ of reducible fibres. This
is not true in general: for example, we exhibit two fibrations on the
Fermat quartic surface with different values of $n(f)$, see Remark
\ref{rem: Fermat quartic}.

Taking the contrapositive of Theorem \ref{thm: main}, we obtain the
following numerical criterion for non-polyhedrality. Since the
pseudo-effective cone of a Mori dream space is polyhedral, the
criterion also provides an obstruction to the Mori dream property, see Corollary \ref{cor:crit}.

\begin{customcor}{B}\label{cor:main_corollary}
Let $S$ be a smooth projective surface endowed with a fibration $f \colon S \to C$. If
\begin{equation} \label{eq:main_corollary}
\rho(S) > n(f)+2, 
\end{equation}
then $\overline{\operatorname{Eff}}(S)$ is non-polyhedral. In particular, $S$ is not a Mori dream space. 
\end{customcor} 

Although this article focuses mainly on surfaces of general type, our
result can also be useful in other contexts. For instance, it yields a
finiteness result for the Mordell--Weil group of the generic fibre of
an elliptic fibration, see Corollary \ref{cor: Mordell-Weil}.

We apply this circle of ideas to several surfaces and families of
surfaces of general type with $q=0$, namely:
\begin{itemize}
\item certain double covers of Hirzebruch surfaces, including 
Horikawa surfaces for every integer $p_g \geq 3$, that is, minimal
surfaces of general type satisfying $K^2=2p_g-4$
(Theorem \ref{thm:covering_Hirzebruch_non_Mori_Dream} and
Corollary \ref{cor:Horikawa});
\item Fermat surfaces in $\mathbb{P}^3$ of every degree $d \geq 4$
(Theorem \ref{thm: Fermat_surf_with_d_geq_4_are_not_MD});
\item certain product-quotient surfaces, namely minimal resolutions of
quotients $(C_1 \times C_2)/G$, where a finite group $G$ acts
diagonally on a product of curves
(Theorems \ref{thm:surface_JS_non_mori_dream},
\ref{thm:surface_WS_non_mori_dream}, and
\ref{thm:not_Mori_Dream_PQ});
\item the Craighero-Gattazzo surface, a simply connected minimal
surface of general type with $p_g=0$ and $K^2=1$
(Theorem \ref{thm: Craighero-Gattazzo_non_Mds}).
\end{itemize}
In each case, we compare the Picard number with the weighted number of
reducible fibres of a suitable fibration and apply Corollary
\ref{cor:main_corollary}. We thereby prove that the pseudo-effective
cone is non-polyhedral and, consequently, that the surface is not a
Mori dream space.

Our work was motivated by a question posed by Keum and Lee
\cite{KL19}, who asked whether there exists a minimal surface of
general type with $p_g=q=0$ that is not a Mori dream space. We show
that the Craighero-Gattazzo surface provides an affirmative answer to
this question. Other examples were recently constructed in
\cite{CCCK}; all of them have Picard number $2$, and therefore have
polyhedral pseudo-effective cones. By contrast, the pseudo-effective
cone of the Craighero-Gattazzo surface is non-polyhedral, so our
example is of a substantially different nature.

We note that the assumption $p_g=0$ in \cite{KL19} is intended mainly
to simplify the computation of the Picard number. The examples with
$p_g>0$ constructed in this paper are therefore also relevant to the
broader problem.

We conclude the introduction by describing the organization of the
paper. Section \ref{sec:preliminaries} recalls some preliminary
results on Mori dream surfaces. Section
\ref{sec: criterion_non_poly} contains the proof of Theorem
\ref{thm: main} and Corollary \ref{cor:main_corollary}, together with
the application to elliptic fibrations. In the subsequent sections, we
prove that several families of surfaces of general type have
non-polyhedral pseudo-effective cones by studying the reducible fibres
of suitable fibrations.

{\bf Notation and conventions} We work over the field $\mathbb{C}$ of complex numbers.

Let $X$ be a normal, $\mathbb{Q}$-factorial projective variety, and let
$K_X$ be a canonical Weil divisor. We define the geometric genus and
the irregularity of $X$ by
\begin{equation}
p_g(X):=h^0\left( X, \, \mathcal{O}_X(K_X)\right)
       =h^0(X, \, \omega_X),
\qquad
q(X):=h^1(X, \, \mathcal{O}_X).
\end{equation}

We denote numerical equivalence of divisors by $\equiv$.

A \emph{fibration} is a surjective morphism
$f\colon X\to Y$ with connected fibres between normal projective
varieties such that $\dim Y<\dim X$.

\section{Preliminaries: Mori dream surfaces}\label{sec:preliminaries}
If $X$ is a normal $\mathbb{Q}$-factorial projective variety, we set
\begin{equation}
N^1(X)_{\mathbb Q}:=(\Pic(X)/\equiv)\otimes_{\mathbb Z} \mathbb Q, \qquad \makebox{and} \qquad N^1(X):=
\bigl(\operatorname{Pic}(X)/{\equiv}\bigr)
\otimes_{\mathbb{Z}}\mathbb{R}.
\end{equation}

The dimension of $N^1(X)$ as a real vector space is the
\emph{Picard number} $\rho(X)$ of $X$. We denote by $N_1(X)$ the real
vector space of $1$-cycles modulo numerical equivalence.

Since $X$ is $\mathbb{Q}$-factorial, every Weil divisor determines a
class in $N^1(X)$. The \emph{effective cone} of $X$ is the convex cone
generated by the numerical classes of effective Weil divisors:
\begin{equation}
\operatorname{Eff}(X):=
\left\{
\sum_{i=1}^{r}a_i[D_i]
\; \; \; | \; \; \;
a_i\in\mathbb{R}_{\geq0},\; \; 
D_i\ \text{is an effective Weil divisor}
\right\}
\subset N^1(X).
\end{equation}

The \emph{pseudo-effective cone}
$\overline{\operatorname{Eff}}(X)$ is the closure of
$\operatorname{Eff}(X)$ in $N^1(X)$.

The \emph{Mori cone} $\overline{\operatorname{NE}}(X)\subset N_1(X)$
is the closure of the cone generated by numerical classes of effective
curves. If $\dim X=2$, the identification of Weil divisors with
$1$-cycles induces a natural isomorphism
\begin{equation}
N^1(X)\simeq N_1(X),
\end{equation}

under which $\overline{\operatorname{Eff}}(X)$ is identified with
$\overline{\operatorname{NE}}(X)$.

The \emph{nef cone} $\operatorname{Nef}(X)\subseteq N^1(X)$ is the cone
consisting of the numerical classes of nef divisors.

A $\mathbb{Q}$-Cartier divisor $D$ on $X$ is called \emph{semi-ample}
if some positive integral multiple of $D$ is Cartier and its complete
linear system is base-point-free. The \emph{semi-ample cone}
$\operatorname{SAmp}(X)$ is the cone generated by the numerical classes
of semi-ample divisors.

A cone in $N^1(X)$ is called \emph{polyhedral} if it is the positive
hull of finitely many vectors. In particular, every polyhedral cone is
closed.

For the following characterization of Mori dream spaces, see
\cite[Definition 1.10 and Proposition 2.9]{HuKeel00}.

\begin{definition}\label{def:Mori_dream_space}
A normal projective variety \(X\) is a \emph{Mori dream space} if \(X\) is
\(\mathbb{Q}\)-factorial, the equality
\begin{equation}
\operatorname{Pic}(X)_{\mathbb{Q}}=N^1(X)_{\mathbb{Q}},
\end{equation}
holds and the Cox ring of  $X$ is finitely generated.
\end{definition}

For a smooth complex projective variety $X$, the condition
$\operatorname{Pic}(X)_{\mathbb{Q}}=N^1(X)_{\mathbb{Q}}$ is equivalent to
$q(X)=0$. Thus,  for the smooth projective surfaces considered in this
paper,  Definition \ref{def:Mori_dream_space} amounts to requiring that
$q(S)=0$  and that the Cox ring of  $S$ be finitely generated.

For every smooth projective surface $S$, we have the inclusions
\begin{equation}
\operatorname{SAmp}(S)
\subseteq \operatorname{Nef}(S)
\subseteq \overline{\operatorname{Eff}}(S),
\end{equation}
see \cite[Proposition 1.2]{AL2011}. Moreover, with respect to the
intersection pairing, $\operatorname{Nef}(S)$  is the dual cone of
$\overline{\operatorname{Eff}}(S)$. Consequently, if
$\overline{\operatorname{Eff}}(S)$  is polyhedral, then
$\operatorname{Nef}(S)$  is polyhedral as well. In particular, if
$\operatorname{Eff}(S)$  is polyhedral, then it is closed and hence $\operatorname{Eff}(S)=\overline{\operatorname{Eff}}(S)$, so $\operatorname{Nef}(S)$ is also polyhedral.

The following result is \cite[Corollary 2.6]{AHL2010}.

\begin{theorem}\label{thm:Mori_dream_surface}
Let  $S$ be a smooth projective surface with $q(S)=0$. Then $S$  is a
Mori dream space if and only if  $\operatorname{Eff}(S)$ is polyhedral and $\operatorname{Nef}(S)=\operatorname{SAmp}(S)$.
\end{theorem}

\section{The main theorem}\label{sec: criterion_non_poly}

We begin by defining the integer  $n(f)$, which measures the number of reducible fibres of  $f$, counted with a suitable weight.

\begin{definition}\label{def: N}
Let \(f \colon S \to C\) be a fibration from a smooth projective surface \(S\) onto a smooth projective curve \(C\). For each \(p \in C\), let \(k_p(f)\) denote the number of irreducible components of the set-theoretic fibre
$f^{-1}(p)$, and set
\begin{equation}\label{eq:n(f)}
n(f):=\sum_{p \in C}\bigl(k_p(f)-1\bigr).
\end{equation}
\end{definition}

By generic smoothness, all but finitely many fibres of  $f$ are smooth. Since the fibres of  $f$ are connected, every smooth fibre is irreducible; hence the sum in \eqref{eq:n(f)} is finite. Notice also that an irreducible fibre contributes zero to $n(f)$, even if it is non-reduced. Thus the non-zero terms in \eqref{eq:n(f)} correspond precisely to the reducible fibres of  $f$; in particular,  $n(f)=0$ if and only if every fibre of  $f$ is irreducible.

We first prove one inequality in Theorem \ref{thm: main}, namely $\rho(S) \geq n(f)+2$; this inequality holds without any assumption on $\overline{\operatorname{Eff}}(S)$. 

Let  $F$ be a smooth fibre of  $f$, and denote its numerical class in $N^1(S)$ by  $e=[F]$. We also denote by $e_1,\ldots ,e_m$ the numerical classes of all the irreducible components of the reducible fibres of  $f$, where $m=0$  if  $f$ has no reducible fibres.

\begin{lemma}\label{lem: N+1}
With the notation above, we have
\begin{equation} 
\dim_{\mathbb{R}}\,\operatorname{span}(e, \, e_1,\dots, e_m) = n(f)+1.
\end{equation}
Consequently, we get
\begin{equation} \label{eq:main_inequality}
\rho(S) \geq n(f)+2.
\end{equation}
\end{lemma}
\begin{proof} 
Let us show how to extract from the set $\{e, \,e_1,\dots, e_m\}$ a basis of the corresponding vector subspace. Each reducible fibre $f^{-1}(p)$ of $f \colon S \to C$ gives a linear relation in $N^1(S)$ involving $e$ and the classes of components of $f^{-1}(p)$. If $r$ is the number of reducible fibres of $f$, we have $r=m-n(f)$ linear relations and so we are left with
\begin{equation}
m+1-r = m+1-(m-n(f))=n(f)+1
\end{equation} 
generators,  which are linearly independent by the classical lemma of Zariski \cite[Chap. III, Lemma 8.2]{BPHV04}.

Finally, every class in $\operatorname{span}(e, \, e_1,\ldots,e_m)$
is orthogonal to $e$, whereas $he>0$  for every ample class
$h$. Hence this span is a proper subspace of $N^1(S)$, and
\eqref{eq:main_inequality} follows.
\end{proof}

In order to prove the reverse inequality we need the following variant of a result by Rabindranath \cite{Ra19}*{Proposition 2.1}, for which we give here a slightly different proof.
\begin{proposition}\label{prop: e^square=0_and_extremal_implies_not_poly}
Let $S$ be a smooth projective surface. Suppose that one can find three non-zero classes $e, \,v, \, h \in N^1(S)$, with $e$ pseudo-effective and $h$ ample, such that
\begin{equation} \label{eq:e_v_h}
e^2=ev=hv=0
\end{equation}
and moreover
\begin{equation}\label{eq:line_intersection}
(e+ \mathbb{R}v) \cap \overline{\operatorname{Eff}}(S) = \{e\}.
\end{equation}

Then $\overline{\operatorname{Eff}}(S)$ is non-polyhedral.
 \end{proposition}

\begin{proof}
Since $e^2 =0$ and $h^2 >0$, the classes $e, \, h$ generate a $2$-dimensional subspace  $\operatorname{span}(e, \, h) \subset N^1(S)$. Since Kleiman's criterion \cite{K66} implies $he>0$, a straightforward calculation shows that the radical of the restriction of the intersection form to $\operatorname{span}(e, \, h)$ is trivial. But $v \in \operatorname{span}(e, \, h)^{\perp}$, so $e, \, v, \, h$ are linearly independent and the subspace $W:=\operatorname{span}(e, \, v, \, h)  
 \subset N^1(S)$ has dimension $3$. By the Hodge Index theorem the restriction of the intersection form to $h^{\perp}$ is negative definite: then, as $\dim (W \cap h^{\perp})=2$, the intersection form on $W$ has signature $(1, \, 2)$.

After rescaling  $h$, we may assume that $h^2=1$, and we can choose an orthogonal basis $\{h, \, w_1, \,  w_2\}$ of $W$ such that $w_1^2=w_2^2=-1$. With respect to this basis, setting $w=xh + y_1w_1+y_2w_2$, the isotropic cone $\{w^2=0 \} \subset W$ has equation $x^2-y_1^2-y_2^2=0$. Thus, the intersection of this cone  with the affine plane 
\begin{equation}
H:=\{ w \in W \; \; | \; \; (w-e)h=0\},
\end{equation}
 which is orthogonal to the $x$-axis and does not contain the origin, is the circle $C$ of equation
\begin{equation} 
y_1^2+y_2^2=(he)^2.
\end{equation} 
Let  $U$ be the bounded open region enclosed by  $C$; then every  $w\in U$ satisfies
\begin{equation} \label{eq:inequalities_w}
w^2>0
\quad\text{and}\quad
h w >0.
\end{equation}
Choose a sequence of rational classes  $w_i\in N^1(S)_{\mathbb Q}$
converging to $w$. Since both inequalities in \eqref{eq:inequalities_w} are open conditions, for
sufficiently large $i$ we have
\begin{equation}
w_i^2>0
\quad\text{and}\quad
h w_i>0.
\end{equation}
By the Riemann-Roch theorem
\cite[Chapter V, Corollary 1.8]{Har83}, a sufficiently divisible integral multiple of a divisor representing
$w_i$ is linearly equivalent to an effective divisor. Therefore $w_i\in\operatorname{Eff}(S)$ and so, passing to the limit,  $w\in\overline{\operatorname{Eff}}(S)$. It follows that
\begin{equation}
\overline{U}\subseteq\overline{\operatorname{Eff}}(S).
\end{equation}

Finally, if  $\overline{\operatorname{Eff}}(S)$  were polyhedral, then
\begin{equation}
P:=H\cap\overline{\operatorname{Eff}}(S)
\end{equation}
would be a (possibly unbounded) convex polygonal region containing
$\overline U$. By \eqref{eq:line_intersection}, the affine line $e + \mathbb{R}v$, which is contained in $H$, intersects $P$ only at $e$. Thus, the point $e$ would be a vertex of  $P$. However, $e\in C=\partial U$, and a convex polygonal region containing  the closed disk $\overline U$ cannot have a vertex on the smooth boundary  $C$. Therefore $\overline{\operatorname{Eff}}(S)$  is non-polyhedral.
\end{proof}

We can complete the proof of  Theorem \ref{thm: main} by showing the inequality $\rho(S) \leq n(f)+2$. 
Our argument refines the proof of a weaker inequality due to Rabindranath \cite[Theorem 12, p. 36]{RabThesis}.

\begin{theorem}\label{thm: criterion}
	Let $f \colon S\to C$ be a fibration of a smooth complex projective surface $S$ onto a smooth projective curve $C$. If $\overline{\operatorname{Eff}}(S)$ is polyhedral, then equality in \eqref{eq:main_inequality} holds, namely
	\begin{equation} 
	\rho(S) = n(f)+2.
	\end{equation}
	\end{theorem}
\begin{proof}
It is enough to prove that, if $\rho(S)> n(f)+2$, then $\overline{\operatorname{Eff}}(S)$ is non-polyhedral (see Lemma \ref{lem: N+1}).

Consider the pushforward map $f_*\colon N^1(S) \simeq  N_1(S)\to  N_1(C) \simeq \mathbb R$. An ample class intersects the fibre $e$ positively,  so $f_*$ is non-zero and $\ker \, f_* \subset N^1(S)$ has codimension $1$.  

Any irreducible component of a fibre is contracted by $f$, so we have the inclusion 
\begin{equation}\label{eq: inclusion_subspaces}
	\operatorname{span} (e, \, e_1,\dots, e_m) \subseteq \ker f_*.
\end{equation}

By Lemma \ref{lem: N+1}, if $\rho(S)> n(f)+2$ then the inclusion in \eqref{eq: inclusion_subspaces} is strict, because
\begin{equation}
\dim \ker f_* = \rho(S)-1   >  n(f)+1=\dim_{\mathbb{R}}\,\operatorname{span}(e, \, e_1,\dots, e_m).
\end{equation}
This means that we can find $v \in \ker f_*$ such that 
\begin{equation} \label{eq:v_notin_span}
v \notin \operatorname{span}(e, \, e_1, \, \ldots, e_m). 
\end{equation} 
 Since $v \in \ker f_*$, we have $ev=0$. Let us fix an ample divisor $h$; then, up to adding to $v$ a suitable multiple of $e$, we may assume $hv=0$. 

Let us consider now the affine line $e+ \mathbb{R}v \subset \ker f_*$.  A result of Debarre, Jiang, and Voisin  \cite{DJV13}*{Theorem 5.7} implies that
\begin{equation} \label{eq:debarre-Jiang-Voisin}
\overline{\operatorname{Eff}}(S) \cap \ker f_* = \RR_{\ge 0} \cdot e+ \RR_{\ge 0}\cdot e_1 + \ldots + \RR_{\ge 0} \cdot  e_m,
\end{equation}
and so, because of \eqref{eq:v_notin_span}, we have
\begin{equation}
(e+ \mathbb{R}v) \cap \overline{\operatorname{Eff}}(S) = \{e\}.
\end{equation}
Summing up, the three classes $e, \, v, \, h \in N^1(S)$ satisfy all the assumptions in Proposition \ref{prop: e^square=0_and_extremal_implies_not_poly}, hence $\psEff(S)$ is non-polyhedral.
\end{proof}

Corollary \ref{cor:main_corollary} follows now easily:
\begin{corollary} \label{cor:crit}
Let $S$ be a smooth projective surface endowed with a fibration $f \colon S \to C$. If
\begin{equation*}
\rho(S) > n(f)+2 
\end{equation*}
then $\overline{\operatorname{Eff}}(S)$ is non-polyhedral. In particular, $S$ is not a Mori dream space. 
\end{corollary} 
\begin{proof} 
Theorem \ref{thm: criterion} together with the assumption $\rho(S)>n(f)+2$  implies that 
$\psEff(S)$ is non-polyhedral. The last claim follows from Theorem \ref{thm:Mori_dream_surface}.
\end{proof}

\begin{remark} \label{rmk:cone_round_at_e}
If $f \colon S \to C$ is a fibration, then the class $e$ of a general fibre
lies on the boundary of $\overline{\operatorname{Eff}}(S)$. Indeed, $e$ is
represented by a smooth irreducible curve and $e^2=0$, so we can apply
\cite[Lemma 1.22]{KM98}.

Moreover, if $\rho(S)>n(f)+2$, then, in the notation of the proof of
Proposition \ref{prop: e^square=0_and_extremal_implies_not_poly}, the affine
line 
\begin{equation}
e+\mathbb{R}v = H \cap e^{\perp}
\end{equation}
 is a supporting line of the closed disk $\overline{U}$,
tangent to its boundary at $e$. Since
\begin{equation}
\overline{U}\subset\overline{\operatorname{Eff}}(S)
\quad\text{and} \quad
(e+\mathbb{R}v)\cap\overline{\operatorname{Eff}}(S)=\{e\},
\end{equation}
the affine section $H\cap\overline{\operatorname{Eff}}(S)$ is not locally
polygonal at $e$. Thus the argument detects the ``roundedness'', and hence
the non-polyhedrality, of $\overline{\operatorname{Eff}}(S)$ locally at $e$.

Notice, however, that $\mathbb{R}_{\geq 0}\cdot e$ is an extremal ray of
$\overline{\operatorname{Eff}}(S)$ if and only if $n(f)=0$.
\end{remark}

\begin{remark} \label{rmk:criterion_necessary_non_sufficient}
The converse of Theorem \ref{thm: main} does not hold:  there exist
fibrations  $f\colon S\to C$ such that $\rho(S)= n(f)+2$  and nevertheless  $\psEff(S)$  is non-polyhedral. An explicit example is provided by the Fermat quartic surface $S_4\subset\mathbb P^3$, endowed with the
fibration described in Remark  \ref{rem: Fermat quartic}.
\end{remark}

\vspace{0.28cm}

We conclude this section with a direct consequence of  Theorem \ref{thm: criterion} concerning the Mordell--Weil group of an elliptic fibration.
\begin{corollary}\label{cor: Mordell-Weil}
Let $S$ be a smooth projective surface with $q(S)=0$ and let $f \colon S \to C$ be an elliptic fibration with a section. If $\overline{\operatorname{Eff}}(S)$ is polyhedral, then the Mordell--Weil group $\operatorname{MW}(f)$ of the generic fibre of $f$ is finite. 
\end{corollary}
\begin{proof}
By Theorem \ref{thm: criterion}, we have $\rho(S)=n(f)+2$. 
On the other hand, the Shioda-Tate formula
\cite[Corollary 6.13]{SS} gives
\begin{equation} \label{eq:Shioda-Tate}
\rho(S)=n(f)+2+
\operatorname{rank}\,\operatorname{MW}(f),
\end{equation}
hence  $\operatorname{rank} \, \operatorname{MW}(f)=0$. Since the
Mordell-Weil group is finitely generated \cite[Theorem 6.1]{SS}, it is finite.
\end{proof}

\begin{remark}
After the first version of this article was posted on arXiv, we became
aware of several related results and observations:
\begin{itemize}
\item C. Casagrande remarked that \cite[Corollary 2.16]{Casagrande} gives the conclusion of
Theorem \ref{thm: main} under the stronger assumption that $S$ is a
Mori dream space;
\item when the base of the fibration is  $\mathbb P^1$, the equality
$\rho(S)=n(f)+2$  was already obtained in
\cite[Remark 2.4]{CLTU};
\item as A. Laface pointed out to us, a different proof of
Theorem \ref{thm: main} can be deduced from \cite[Proposition 2.3]{CLTU};
\item under the additional assumption that  $\kappa(S)=1$, a different
proof of Corollary \ref{cor: Mordell-Weil} was recently given in
\cite[Proposition 3.1]{Laface-Li-Liu26};
\item M. Sch\"{u}tt pointed out to us that, even when $f$  does not admit
a section, under the remaining hypotheses of
Corollary \ref{cor: Mordell-Weil} the Mordell-Weil group of its relative
Jacobian is finite.
\item  Some connected results on fibred elliptic surfaces has been announced in \cite{Keum}.
\end{itemize}
\end{remark}

\section{Irregular examples with non-polyhedral \texorpdfstring{$\psEff(S)$}{pseudo-effective cone}:
self-products of curves and diagonal double Kodaira surfaces}\label{subsec:first-examples}

As a first application of the theory developed above, we get the following result due to Rabindranath \cite{Ra19}*{Theorem 1.1}.

\begin{proposition} \label{prop:CxC}
Let $C$ be a smooth curve with $g(C) \geq 1$. Then $\overline{\operatorname{Eff}}(C \times C)$ is non-polyhedral.
\end{proposition}
\begin{proof}
If  $f \colon C \times C \to C$ is  the projection onto one of the factors, then $n(f)=0$. Moreover, we have  $\rho(C \times C) \geq 3$, since the classes of the fibres of the two projections and of the diagonal
$\Delta\subset C\times C$  are linearly independent in $N^1(C \times C)$ for $g(C)\geq 1$. So the result follows from Corollary \ref{cor:main_corollary}.
\end{proof}

We can use the same technique to prove that some double Kodaira surfaces (i.e., surfaces endowed with two smooth, non-isotrivial fibrations, see \cite{CatRol09}) also have a non-polyhedral pseudo-effective cone. The following example can be found in \cite{PolSab22}, see also \cite{PolSab26} for related constructions. Let $C$ be a smooth curve with $g(C)=2$ and let $\Delta \subset C \times C$ be the diagonal. Then there exists a $G$-cover 
\begin{equation} \label{eq:extra-special cover}
\pi \colon S \to C \times C,
\end{equation}
where $G$ is an extra-special group of order $2^5=32$. The cover $\pi$
is branched precisely along $\Delta$, with branching order $2$. Composing $\pi$ with the two natural projections  of $C \times C$, we get two smooth, non-isotrivial fibrations $f_i \colon S \to C$. The surface  $S$ is called a \emph{diagonal double Kodaira surface}.
\begin{proposition}
Let $S$ be a diagonal double Kodaira surface as described above. Then  $\overline{\operatorname{Eff}}(S)$ is non-polyhedral.
\end{proposition}
\begin{proof}
Writing $F_1$, $F_2$ for the fibres of the two fibrations on $S$, and calling $D$ the ramification locus of $\pi$, namely, the  smooth curve $D \subset S$ such that $\pi^* \Delta = 2D$, we get
\begin{equation}
F_1^2=F_2^2=0,  \quad D^2=-16, \quad F_1D=F_2D=16, \quad F_1F_2=32. 
\end{equation}
Thus the Gram matrix of the restriction of the intersection form to the subspace of $N^1(S)$ spanned by $F_1, \, F_2, \, D$ is the matrix
\begin{equation}
\begin{pmatrix}
0 & 32 & 16 \\
32 & 0 & 16 \\
16 & 16 & -16 \\
\end{pmatrix},
\end{equation} 
whose determinant is $32768$. So $F_1, \, F_2, \, D$  are linearly independent in $N^1(S)$, hence we get $\rho(S) \geq 3$. On the other hand, since $f_1 \colon S \to C$ is a smooth fibration, we have $n(f_1)=0$. Thus the claim follows from Corollary \ref{cor:main_corollary}.
\end{proof}
Similar examples can be constructed using a smooth curve $C$ of any genus $g \geq 2$ and an extra-special group $G$ of order $p^{2g+1}$, with  $p$ dividing $g+1$; see \cite{CaPol21} for more details.

\begin{remark} \label{rmk:roundness_again} Since all the fibrations considered in this section satisfy $n(f)=0$ and
the corresponding surfaces have Picard number at least $3$, the argument
of Proposition \ref{prop: e^square=0_and_extremal_implies_not_poly} shows
that the class $e$ of a fibre provides an extremal ray of $\psEff(S)$ such that $e^2=0$, and that $\psEff(S)$ is ``round'' near this class; see   Remark \ref{rmk:cone_round_at_e}.
\end{remark}

\section{Double covers of ruled surfaces and Horikawa surfaces}

In this section, we show how to apply Corollary  \ref{cor:main_corollary} to double covers of ruled surfaces. Let $n$ be a non-negative integer and consider the Hirzebruch ruled surface $\mathbb F_n$, endowed with the ruling 
$\mathbb{F}_n \to \mathbb{P}^1$. We have 
\begin{equation} \label{eq:Neron_Severi_Hirzebruch}
N^1(\mathbb{F}_n)=\mathbb{R}[F] \oplus \mathbb{R}[C_0],
\end{equation}
where $F$ is a  fibre of the ruling and $C_0$ is a section with $C_0^2=-n$, which is unique when $n \geq 1$. We take a double cover
\begin{equation} \label{eq:cover_Hirzebruch}
\pi \colon S \to \mathbb{F}_n,
\end{equation}
branched over a smooth curve $B \subset \mathbb{F}_n$ that is not supported on a union of fibres of the
ruling. Finally, we consider the fibration $f \colon S \to \PP^1$, obtained by composing $\pi \colon S \to \mathbb{F}_n$ with the ruling.

We choose the smooth curve $B$ in order to make the Picard number of the corresponding surface $S$ sufficiently large by applying the following 
\begin{lemma} \label{lem:Picard_number_Cover_F_n}
Assume that the branch divisor $B$ does not contain $C_0$ and that $B$  is everywhere tangent to $C_0$, namely, that the restriction $B|_{C_0}$ yields an effective divisor $($possibly the zero divisor$)$ which is $2$-divisible in $\operatorname{Div}(C_0)$. If $C_0B=0$, assume in addition $n>0$. Then $\rho(S) \geq 3$.
\end{lemma}
\begin{proof}
The assumption that $B$ is everywhere tangent to $C_0$ implies that the
morphism induced by $\pi$ on the normalization of $\pi^*C_0$ is an
\'{e}tale double cover of $C_0$. Since $C_0\simeq\mathbb{P}^1$ has trivial
fundamental group, this cover is disconnected. Therefore  $\pi^*C_0=C_1+C_2$, where $C_1$, $C_2$ are smooth rational curves in $S$ and  $C_1  C_2 = \frac{1}{2} C_0 B$.

The double cover involution on $S$ exchanges $C_1$ and $C_2$, so that
\begin{equation}
C_1^2=C_2^2=\frac12 (C_1+C_2)^2-C_1C_2=C_0^2-\frac12 C_0B.
\end{equation}

If $e=\pi^*F$ is the class of a fibre of $f \colon S \to \mathbb{P}^1$, then the projection formula yields  
\begin{equation}
(C_1+C_2)e = (\pi^*C_0) e = C_0 (\pi_*e) = C_0  (2F) =2.
\end{equation}
So, since the cover involution preserves $e$, we get  $C_1 e = C_2 e =1$.  Thus the Gram matrix of the restriction of the intersection form to the subspace of $N^1(S)$ spanned by $C_1, \, C_2, \, e$ is the matrix
\begin{equation}
\begin{pmatrix}
 C_0^2-\frac{1}{2} C_0 B& \frac12 C_0B & 1\\
 \frac12 C_0 B& C_0^2-\frac12 C_0 B& 1\\
  1 & 1 & 0\\
\end{pmatrix}
\end{equation} 
whose determinant is $2(C_0B-C_0^2) = 2(C_0B+n) >0$. Therefore $C_1, \, C_2, \, e$ are linearly independent in $N^1(S)$, and hence $\rho(S) \geq 3$. 
\end{proof}

\begin{lemma} \label{lem:polynomial-not-a-square}
Let $k$ be an algebraically closed field with $\operatorname{char}(k) =0$, and let $\alpha, \, \beta \in k$ such that $(\alpha, \, \beta) \neq (0, \, 0)$. For all integers $b \geq 2$, the polynomial
\begin{equation*}
p(x)=x^{2b} + \alpha x^{2b-1} + \beta \in k[x]
\end{equation*}
has at least one simple root in $k$. In particular, $p(x)$ is not a square in $k[x]$.
\end{lemma}
\begin{proof}
If $\beta =0$ then $p(x)=x^{2b} + \alpha x^{2b-1} = x^{2b-1}(x+\alpha)$ has $-\alpha$ as a simple root.  If $\alpha=0$ then $p(x)=x^{2b} + \beta$ which has only simple roots.
Thus we can assume $\alpha \neq 0$ and $\beta \neq 0$. The derivative of $p(x)$ is the polynomial
\begin{equation} \label{eq:derivative_p}
p'(x)= 2bx^{2b-1}+(2b-1) \alpha x^{2b-2} = x^{2b-2}(2bx + (2b-1) \alpha),
\end{equation}
which has the two distinct roots $0$ (with multiplicity $2b-2$) and $\gamma := \frac{(1-2b) \alpha}{2b}$ (with multiplicity $1$). 

As $p(0)= \beta \neq 0$, the only possible multiple root of $p(x)$ is  $\gamma$.  Moreover, since $\gamma$ is a simple root of $p'(x)$, it can have multiplicity at most  $2$ as a root of $p(x)$. Thus $p(x)$ has at least $2b-2  \geq 2$ simple roots.
\end{proof}

In the sequel,   we identify  $\FF_n$ with  a  toric variety $\CC^4// (\CC^*)^2$ as follows. Take 
coordinates $t_0, \, t_1, \, x_0, \, x_1$ on $\CC^4$,  and let $(\CC^*)^2$ act by the weight matrix
\begin{equation} \label{eq:weight_matrix}
\begin{pmatrix}
t_0&t_1&x_0&x_1\\
1&1&0&n\\
0&0&1&1
\end{pmatrix}
\end{equation}
choosing  $(t_0, \, t_1)\cap (x_0, \, x_1)$ as the  irrelevant ideal. Thus  $t_0$ and $t_1$ have bidegree $(1, \, 0)$, whereas $x_0$ has bidegree $(0, \, 1)$ and $x_1$ has bidegree $(n, \, 1)$.

This identifies the vector spaces $H^0(\FF_n,  \, \mathcal{O}_{\FF_n}(aF+b C_0))$ with the homogeneous polynomials in the variables $t_0, \, t_1, \, x_0, \, x_1$ of  bidegree $(a, \, b)$ with respect to the grading defined by the weight matrix \eqref{eq:weight_matrix}.\footnote{For the general theory of homogeneous coordinates on toric varieties, see \cite[Chapter 5]{CLS11}.}

For simplicity, we write $H^0(\mathbb{F}_n, \, \mathcal{O}(a, \, b))$ instead of  $H^0(\FF_n,  \, \mathcal{O}_{\FF_n}(aF+b C_0))$. So the ruling is given by $H^0(\mathbb{F}_n, \, \mathcal{O}(1, \,0)) =\operatorname{span}(t_0, \, t_1)$, whereas $C_0$ is defined by the zero locus of $x_0 \in H^0(\mathbb{F}_n, \, \mathcal{O}(0,\,1))$.

\begin{lemma}\label{lem:fibre_irreducible}
For every pair of integers $a\geq0$ and $b\geq2$, we can choose a  branch curve $B\in\left|\mathcal{O}(2a+2bn, \, 2b)\right|$ so that the following two conditions hold$:$
\begin{itemize}
\item[$(1)$] $B$ is smooth and everywhere tangent to $C_0;$
\item[$(2)$] every fibre of $f\colon S\to\mathbb{P}^1$ is irreducible.
\end{itemize}
\end{lemma}
\begin{proof}
Fix integers $a\geq0$ and $b\geq2$, and choose pairwise distinct nonzero
complex numbers $c_1,\ldots,c_a$. We consider global sections $p_B\in H^0\left(\mathbb{F}_n, \,
\mathcal{O}(2a+2bn, \, 2b) \right)$
of the form
\begin{equation}\label{eq:pB}
p_B
=x_1^{2b}\prod_{i=1}^a(t_0-c_it_1)^2
+x_0g(t_0, \, t_1; \, x_0, \,x_1),
\end{equation}
where
\begin{equation}
g\in H^0\left(\mathbb{F}_n, \,
\mathcal{O}(2a+2bn,2b-1)\right).
\end{equation}
The zero locus $B$ of $p_B$ is an effective divisor whose class is $2$-divisible in $\operatorname{Pic}(\mathbb{F}_n)$, and so there is a double cover
\begin{equation}
\pi\colon S\longrightarrow\mathbb{F}_n
\end{equation}
branched along $B$.\footnote{Such a double cover is unique up to
isomorphism, since $\operatorname{Pic}(\mathbb{F}_n)$ has no
$2$-torsion.}

As $g$ varies, the sections in \eqref{eq:pB} form an affine space $\mathcal{L}$, which is an open subset  of the projective linear system
generated by
\begin{equation}
x_1^{2b}\prod_{i=1}^a(t_0-c_it_1)^2
\quad\text{and}\quad
x_0H^0\left(\mathbb{F}_n, \,
\mathcal{O}(2a+2bn, \, 2b-1)\right).
\end{equation}
The base points of $\mathcal{L}$ are precisely
\begin{equation}
P_i=[c_i:1:0:1],\quad 1\leq i\leq a.
\end{equation}
Every $B\in\mathcal L$ satisfies
\begin{equation}
B|_{C_0}=2(P_1+\cdots+P_a),
\end{equation}
and is therefore everywhere tangent to $C_0$. A general member of  $\mathcal{L}$ is
smooth at the points $P_i$, and Bertini's theorem shows that it is also
smooth away from them. Hence a general member of $\mathcal L$ is
smooth, and condition~$(1)$ is satisfied.\footnote{If $a=0$, then $B$
is disjoint from $C_0$ and the corresponding linear system is
base-point free.}

To verify condition~$(2)$, it is enough to choose a curve $B$ that intersects
every fibre of the ruling transversely at some point. Indeed,
an elementary local computation shows that $\pi^*F$ is smooth at the
point lying over such an intersection. If $\pi^*F$ were reducible, its
two components would meet at that point, contradicting the smoothness
of $\pi^*F$ there.

The locus of pairs $(B, \, t)\in\mathcal L\times\mathbb{P}^1$ such that the restriction of $p_B$ to $F_t$ has no simple zero is closed. Since the projection $\mathcal L\times\mathbb{P}^1\longrightarrow\mathcal L$
is proper, its image is closed. We now construct a member of $\mathcal L$ that does not belong to this image.

Consider the section
\begin{equation}
p_B
=x_1^{2b}\prod_{i=1}^a(t_0-c_it_1)^2
+t_1^{2a+n}x_0x_1^{2b-1}
+t_0^{2a+2bn}x_0^{2b}.
\end{equation}
We claim that its zero locus intersects every fibre of the ruling
transversely in at least one point. Fix
$[\bar t_0:\bar t_1]\in\mathbb{P}^1$. The restriction to the
corresponding fibre is
\begin{equation} \label{eq:binary_form}
p_B(\bar t_0, \, \bar t_1; \, x_0, \, x_1)
 =\tau x_1^{2b}+\alpha x_0x_1^{2b-1}+\beta x_0^{2b},
\end{equation}
where
\begin{equation}
\tau=\prod_{i=1}^a(\bar t_0-c_i\bar t_1)^2,\quad
\alpha=\bar t_1^{2a+n},\quad
\beta=\bar t_0^{2a+2bn}.
\end{equation}
Proving our claim is equivalent to showing that the binary form \eqref{eq:binary_form} has a linear factor of multiplicity $1$.

Since all the $c_i$ are nonzero, at most one of  $\alpha, \, \beta,\, \tau$
vanishes.

If $\beta=0$, then $\alpha\tau\neq0$ and
\begin{equation}
p_B=x_1^{2b-1}(\tau x_1+\alpha x_0),
\end{equation}
so the factor $\tau x_1+\alpha x_0$ has multiplicity $1$.

Suppose now that $\beta\neq0$. Setting $x=x_1/x_0$ gives
\begin{equation}
p(x)=\tau x^{2b}+\alpha x^{2b-1}+\beta.
\end{equation}
If $\tau\neq0$, after dividing by $\tau$ we may apply
Lemma~\ref{lem:polynomial-not-a-square}, which shows that $p(x)$ has a
simple root. If $\tau=0$, then $\alpha\beta\neq0$ and
\[
p(x)=\alpha x^{2b-1}+\beta,
\]
which has only simple roots. This proves our claim: every fibre meets $B$ transversely
in  at least one point.

It follows that condition~$(2)$ holds on a nonempty open subset of
$\mathcal L$. Since $\mathcal L$ is irreducible, this open subset
intersects the nonempty open subset parametrizing the smooth members.
Choosing $B$ in their intersection proves the lemma.
\end{proof}

We can now state the main result of this section.
\begin{theorem} \label{thm:covering_Hirzebruch_non_Mori_Dream}
Let $a, \, n\geq0$ and $b\geq2$ be integers, and assume that
$a+n\geq2$. Consider the double cover $\pi \colon S \to \FF_n$, branched along a curve $B \in |2(a+bn)F+2bC_0|$ whose defining polynomial is general among those in the affine space $\mathcal{L}$ defined above. Then $S$  is a smooth minimal surface  such that  $\overline{\operatorname{Eff}}(S)$ is non-polyhedral. 
\end{theorem}
\begin{proof}
By Lemma~\ref{lem:fibre_irreducible}, a general branch curve
$B\in\mathcal L$ is smooth, and hence the corresponding double cover
$S$ is smooth. Moreover, every fibre of
$f\colon S\to\mathbb{P}^1$ is irreducible, so $n(f)=0$. We have
\begin{equation}
C_0 B =C_0 \left( 2(a+bn)F+2bC_0 \right) =2a.
\end{equation}
Thus, if $C_0 B=0$ then $a=0$ and the assumption $a+n\geq2$
implies $n\geq2$. Therefore
Lemma \ref{lem:Picard_number_Cover_F_n} applies and gives
$\rho(S)\geq3$. Consequently,  Corollary \ref{cor:main_corollary} implies that
$\overline{\operatorname{Eff}}(S)$ is non-polyhedral.

Finally, the double cover formula gives
\begin{equation}
K_S=
\pi^*\left( K_{\mathbb{F}_n} + (a+bn)F+bC_0 \right),
\end{equation}
where the divisor inside the pullback can be written as
\begin{equation}
(b-2)(C_0+nF)+(a+n-2)F.
\end{equation}
Since $F$ and $C_0+nF$ are nef and both coefficients are non-negative,
this divisor is nef. This means that  $K_S$ is also nef, hence  $S$ is minimal.
\end{proof}

Let us recall that a \emph{Horikawa surface} is a minimal surface $S$ of general type satisfying  $K_S^2 = 2p_g(S)-4$.\footnote{That is, Horikawa surfaces attain equality in the Noether inequality $K_S^2 \geq 2p_g(S)-4$.}  Horikawa surfaces satisfy $p_g(S)\geq 3$ and $q(S)=0$; we refer to \cite[Chapter VII]{BPHV04} and to the references therein for further details. 

\begin{corollary} \label{cor:Horikawa}
For every integer $p\geq3$, there exists a Horikawa surface $S$ with
$p_g(S)=p$ such that  $\overline{\operatorname{Eff}}(S)$ is non-polyhedral. In particular,
$S$ is not a Mori dream space.
\end{corollary}
\begin{proof}
Set $b=3$ and choose $a, \, n$ and a general branch curve $B$ as in
Theorem \ref{thm:covering_Hirzebruch_non_Mori_Dream}. Writing the linear equivalence class of  $B$ as 
\begin{equation}
B = 2L,\quad L=(a+3n)F+3C_0,
\end{equation}
the double cover formulas give
\begin{equation}
p_g(S)
 =h^0(\mathbb{F}_n, \, K_{\mathbb{F}_n}+L)
 =2a+3n-2
\end{equation}
and
\begin{equation}
K_S^2
 =2 (K_{\mathbb{F}_n}+L)^2
 =4a+6n-8
 =2p_g(S)-4.
\end{equation}
Moreover, $K_S$ is nef by
Theorem \ref{thm:covering_Hirzebruch_non_Mori_Dream}, and it is big
whenever $2a+3n\geq5$.

Now fix an integer $p\geq3$. If $p$ is odd, take
\begin{equation}
n=1,\quad a=(p-1)/2
\end{equation}
whereas if $p$ is even, take
\begin{equation}
n=0,\quad a=(p+2)/2.
\end{equation}
In both cases, $a+n\geq2$, $2a+3n=p+2\geq5$, and
$p_g(S)=p$. Hence $S$ is a Horikawa surface. Its pseudo-effective cone
is non-polyhedral by Theorem \ref{thm:covering_Hirzebruch_non_Mori_Dream}, so $S$ is not a
Mori dream space; see also \cite[Theorem 1.6]{Hor76}.
\end{proof}

\section{Surfaces in \texorpdfstring{$\mathbb{P}^3$}{P3} containing lines and Fermat surfaces} \label{sec:fermat_surfaces}

A natural class of surfaces to which one may try to apply Corollary
\ref{cor:main_corollary} is the class of smooth surfaces
$S\subset\mathbb{P}^3$ containing a line $\ell$.

Let $H$ be the hyperplane class of such a surface $S$. The residual
pencil $|H-\ell|$ has no fixed components, because the only curve common
to two distinct planes containing $\ell$ is $\ell$ itself.

Set $d:=\deg S\geq 2$. By adjunction, we have
$K_S=(d-4)H$. Thus, $K_S\ell=d-4$, and the genus formula gives
$\ell^2=2-d$. Therefore,
\begin{equation}
(H-\ell)^2=H^2-2H\ell+\ell^2=d-2+(2-d)=0.
\end{equation}
Since the pencil has no fixed components, this equality implies that it
has no base points. Hence, it defines a morphism
$f\colon S\to\mathbb{P}^1$. Its fibres are plane curves of degree $d-1$
and are therefore connected, so $f$ is a fibration.

\begin{remark}\label{rem: the_fibres_are_reduced}
	The fibres of $f$ are always reduced. Indeed, let
$\Pi\subset\mathbb{P}^3$ be a plane containing $\ell$, and suppose that
the corresponding fibre contains an irreducible component $\Gamma$ with
multiplicity greater than one. At a general smooth point $p\in\Gamma$,
the plane $\Pi$ coincides with the tangent plane to $S$ at $p$, namely, 
$\Pi$ is tangent to $S$ along $\Gamma$. It follows that the
Gauss map $\gamma \colon S \to (\mathbb{P}^3)^* $ contracts  $\Gamma$  to a point, contradicting the fact that $\gamma$ is finite \cite{Zak}.
 \end{remark} 

In order to estimate $n(f)$, it is useful to recall some standard facts about Milnor numbers of plane curves.
\begin{definition}
Let $C \subset \mathbb{P}^2$ be a reduced curve of degree $d-1$. Writing $\chi_{\operatorname{top}} (C)$ for its Euler characteristic, we define the \emph{total Milnor number of} $C$ as 
\begin{equation} \label{eq:total_Milnor_number}
\mu(C):= \chi_{\operatorname{top}} (C) + (d-1)(d-4).
\end{equation}
\end{definition}

By the standard theory of plane curves, $\mu(C)\ge 0$ and $C$ is smooth if and only if $\mu(C)=0$.
The number $\mu(C)$ measures the difference between the Euler
characteristic of $C$ and that of a smooth plane curve of the same degree.

Write $C=\sum_{i=1}^k C_i$, where the $C_i$ are the irreducible components of $C$. Let $\nu \colon \tilde{C} \rightarrow C$ be the normalization of $C$ and, for all $p \in C$, write  $r_p:=  \# \nu^{-1}(p)$ for  the number of local branches at $p$. Then
\begin{equation}\label{eq: e(C)}
\chi_{\operatorname{top}}(C)= \sum_{i=1}^k (2-2g_i) -\sum_{p \in \Sing(C)} (r_p-1)
\end{equation}
where $g_i$ is the genus of the normalization of $C_i$.

The natural inclusion ${\mathcal O}_C \hookrightarrow \nu_* {\mathcal O}_{\tilde{C}}$ has cokernel sheaf ${\mathcal Q}$ supported on $\Sing(C)$, so we can write $\mathcal{Q}=\oplus_{p \in \Sing (C)} {\mathcal Q}_p$ with ${\mathcal Q}_p$ supported on $p$. Setting $\delta_p:=$ length $ {\mathcal Q}_p$, the residue map of the normalization implies 
\begin{equation}
\delta_p \geq r_p-1.
\end{equation}
Moreover, from the short exact sequence
\begin{equation}
0 \rightarrow {\mathcal O}_C \rightarrow \nu_* {\mathcal O}_{\tilde{C}} \rightarrow {\mathcal Q} \rightarrow 0
\end{equation}
we obtain the formula
\begin{equation}\label{eq: pa}
p_a(C)=p_a(\tilde{C}) +\sum _{p \in \Sing(C)}\delta_p =\sum_{i=1}^k g_i -(k-1) +\sum_{p \in \Sing(C)} \delta_p.
\end{equation}
Then
\begin{equation}\label{eq: mu vs delta}
\mu(C)=\chi_{\operatorname{top}}(C)-(2-2p_a(C))=\sum_{p \in \Sing(C)} \mu_p(C)  \geq \sum_{p \in \Sing(C)}  \delta_p,
\end{equation}
where $\mu_p(C):= 2\delta_p- r_p+1$ equals the \emph{local Milnor number} of  $C$ at  $p$, cf. \cite[Theorem 10.5]{Milnor68}.


\begin{lemma}\label{lem: mu vs k}
Let  $C$ be a reduced plane curve of degree  $d-1$ with  $k$ irreducible components. Then
\begin{equation}\label{eq: mu vs k}
\mu(C) \geq \frac{1}{2} (d-1) (k-1). 
\end{equation}
\end{lemma}
\begin{proof}
Setting $d_i:=\deg C_i$, so that $d-1=\sum_{i=1}^k d_i$, by \eqref{eq: pa} we get
\begin{equation}\label{eq: delta vs di}
\begin{split}
\sum _{p \in \Sing(C)}\delta_p & =\frac{1}{2}(d-2)(d-3)-\sum_{i=1}^k g_i  +k-1  \\ 
&\geq \frac{1}{2}(d-2)(d-3)- \frac{1}{2}\sum_{i=1}^k (d_i-1)(d_i-2)  +k-1=\frac12 \sum_{i \neq j} d_id_j.
\end{split}
\end{equation}
Since
\begin{equation}\label{eq: di vs k}
\frac12 \sum_{i \neq j} d_id_j = \frac12 \sum_{i=1}^k d_i(d-1-d_i) \ge\frac12 \sum_{i=1}^k d_i(k-1)=\frac12 (d-1) (k-1),
\end{equation}
the claim now follows from \eqref{eq: mu vs delta}, \eqref{eq: delta vs di} and \eqref{eq: di vs k}.
\end{proof}

We are now in a position to prove the following 
\begin{theorem}\label{thm: surf with lines not MD}
Let $S\subset {\mathbb P}^3$ be a smooth surface of degree $d \ge 2$ containing a line $\ell$.  If 
\begin{equation}
\rho(S) >  \frac{2(d-2)^2(d+2)}{d-1} +2
\end{equation}
 then $\overline{\operatorname{Eff}}(S)$ is non-polyhedral, and so $S$ is not a Mori dream space.
\end{theorem}
\begin{proof}
We consider the fibration $f \colon S \rightarrow {\mathbb P}^1$ given by $|H- \ell |$. For all $t \in {\mathbb P}^1$, denote by $F_t$ the corresponding fibre. All $F_t$ are plane curves of degree $d-1$. We also denote by $F$ a general fibre.

Since $F$ is a smooth plane curve of degree $d-1$, the definition of
$\mu(F_t)$ gives
\begin{equation}
\mu(F_t)=
\chi_{\operatorname{top}}(F_t)-\chi_{\operatorname{top}}(F),
\end{equation}
and so, by  the Zeuthen-Segre formula \cite{BPHV04}*{Chapter III, Proposition 11.4 and Remark 11.5}, we get
\begin{equation}
\begin{split}
\sum_{t \in  {\mathbb P}^1} \mu (F_t)& = \chi_{\operatorname{top}}(S) -  \chi_{\operatorname{top}}(\mathbb{P}^1)\chi_{\operatorname{top}}(F) \\ & =d^3-4d^2+6d-2(2-(d-2)(d-3)) =(d-2)^2(d+2).
\end{split}
\end{equation}
Furthermore, all fibres $F_t$ are reduced, see Remark \ref{rem: the_fibres_are_reduced}. Therefore we can apply Lemma \ref{lem: mu vs k}, obtaining 
\begin{equation}
\sum_{t \in  {\mathbb P}^1} \mu (F_t) \geq  \frac{1}{2} (d-1)n(f)
\end{equation}
and so 
\begin{equation} \label{eq:n(f)_Fermat}
n(f) \leq  \frac{2(d-2)^2(d+2)}{d-1}. 
\end{equation}
The claim now follows from Corollary \ref{cor:main_corollary}.
\end{proof}

We apply Theorem \ref{thm: surf with lines not MD} to the \emph{Fermat surface} $S_d \subset \mathbb{P}^3$, defined by
\begin{equation} \label{eq:S_d}
x_0^d-x_1^d+x_2^d-x_3^d=0.
\end{equation}

We can  take the line $\ell  \subset S_d$ of equation
\begin{equation}
x_0-x_1=x_2-x_3=0
\end{equation}
and consider the corresponding fibration 
\begin{equation} \label{eq:fibration_on_S_d}
f_d \colon S_d \to \mathbb{P}^1,
\end{equation}
whose fibres $F_t$ are the members of the movable part of the linear system cut out on $S_d$ by the pencil of planes
\begin{equation}
x_0-x_1=t(x_2-x_3), \quad t \in \mathbb{P}^1
\end{equation}
passing through $\ell$.

\begin{theorem}\label{thm: Fermat_surf_with_d_geq_4_are_not_MD}
Let  $S_d\subseteq \mathbb P^3$ be a Fermat surface of degree $d\geq  4$. Then  $\overline{\operatorname{Eff}}(S_d)$ is non-polyhedral, and so $S_d$ is not a Mori dream space. 
\end{theorem}
\begin{proof}
	By \cite[formula 5]{Shioda}, since all the additional terms appearing
there are non-negative, we have
\begin{equation}\label{eq:rho_S_d}
\rho(S_d)\geq 3(d-1)(d-2)+1.
\end{equation}
Thus, by Theorem \ref{thm: surf with lines not MD}, if
\begin{equation}\label{eq: fermat in}
3(d-1)(d-2)+1 > 2 \frac{(d-2)^2(d+2)}{d-1} +2
\end{equation}
then the pseudo-effective cone of $S_d$ is non-polyhedral. After multiplying by $d-1$, \eqref{eq: fermat in} is equivalent to
\begin{equation}
(d-3) (d^2-5d+7) > 0, 
\end{equation}
which is true for all $d \geq 4$.
\end{proof}

The degree bound in Theorem \ref{thm: Fermat_surf_with_d_geq_4_are_not_MD} is sharp: the Fermat
surfaces $S_2$ and $S_3$ are del Pezzo surfaces and hence Mori dream
spaces.

  \begin{remark}\label{rem: Fermat quartic}
It was already known that the pseudo-effective cone
$\overline{\operatorname{Eff}}(S_4)$ of the Fermat quartic
$S_4\subset\mathbb{P}^3$ is non-polyhedral. In fact, $S_4$ is a K3
surface with Picard number $20$, and a result of Shioda and Inose
\cite{Shioda-Inose}*{Theorem 5 in §5} (see also
\cite{Segre}*{§§17--18}) shows that its automorphism group is infinite.
Moreover, $S_4$ contains lines, which are $(-2)$-curves. Therefore, by
results of Pjateckiĭ-Šapiro and Šafarevič, and of Sterk
\cite{Huybrechts}*{Chapter 8, §4, Corollary 4.7}, the surface $S_4$ contains
infinitely many $(-2)$-curves, which provide infinitely many extremal
rays of its pseudo-effective cone.

There is another interesting fibration  $g\colon S_4 \to \mathbb{P}^1$ defined by the rational function
\begin{equation}
t=\frac{x_0^2-x_1^2}{x_2^2-x_3^2}=-\frac{x_2^2+x_3^2}{x_0^2+x_1^2}. 
\end{equation} 
This is an elliptic fibration having six singular fibres, each a cycle of four lines of Kodaira type $I_4$ \cite{RabThesis}*{Example p. 44}. Thus, we have 
\begin{equation}
	n(g)+2=6 \cdot (4-1) +2 = 20= \rho(S_4).
\end{equation}
This shows at once two interesting things:
\begin{enumerate} 
\item  the equality $\rho(S) = n(f)+2$ in Theorem \ref{thm: main} is necessary, but not sufficient, for the polyhedrality of $\overline{\operatorname{Eff}}(S)$;
\item if $\overline{\operatorname{Eff}}(S)$  is non-polyhedral, then $S$ may admit two distinct fibrations with different values of the invariant  $n(f)$. In fact, for the Fermat surface $S_4 \subset \mathbb{P}^3$ we showed that  $n(g)=18$; on the other hand, we have $n(f_4) \le 16$, see \eqref{eq:n(f)_Fermat}. By Theorem \ref{thm: main}, this situation cannot occur  if $\overline{\operatorname{Eff}}(S)$ is polyhedral. 
\end{enumerate}
  \end{remark}

\section{Product-quotient surfaces}
If  $C$ is a smooth curve with $g(C) \geq 1$, we have seen (Proposition \ref{prop:CxC}) that $\psEff(C \times C)$ is non-polyhedral. However, since $q(C \times C)=2g(C) >0$, we do not need to study its pseudo-effective cone in order to conclude that $C\times C$ is not a Mori dream space.

To obtain regular examples, we slightly modify the setting and consider
minimal resolutions of quotients of products of curves by diagonal
actions of finite groups. Such surfaces have been extensively studied
in the literature \cites{Cat00, fano, isogenous, BaCaGrPi12, BaPi12, FG23, Fede25,
Fede24, MiPo10, Po10, Po15} and are known as
\emph{product-quotient surfaces}.

\begin{definition} \label{def:product-quotient_surfaces}
Let $C_1, \, C_2$ be two smooth projective curves and let $G$ be a finite group acting faithfully on both of them. A \emph{product-quotient} surface is the minimal resolution of singularities $\lambda \colon S \to X$ of the \emph{quotient model} $X=(C_1 \times C_2)/G$, where $G$ acts diagonally on the product.
\end{definition}
The surface $S$ comes naturally with two fibrations
\begin{equation}
f_1 \colon S \to C_1/G, \quad f_2 \colon S \to C_2/G,
\end{equation}
whose general fibres we denote by $F_1$ and $F_2$, respectively. 

We refer to  \cite{Ser96}*{Theorem 2.1}, where the author studied the 
fibres of $f_1$ and $f_2$.  In particular, one can easily observe that 
\begin{equation}
F_1 \simeq C_2, \quad  F_2 \simeq C_1  \quad \textrm{and} \quad  F_1 F_2 = |G|,
\end{equation}
and both fibrations $f_1$ and $f_2$ are \emph{isotrivial}. Moreover
\begin{equation} \label{eq:n_product-quotient}
n(f_1)=n(f_2)=l(X)
\end{equation}
where $l(X)$ is the number of irreducible curves contracted by $\lambda \colon S \to X$, see   \cite{BaPi16}*{Definition 2.3}.  Indeed, the reduced
support of every reducible fibre consists of one non-exceptional
component (called the \emph{central component}) together with a finite number of Hirzebruch-Jung strings contracted by $\lambda$ to singular points of $X$.  Since every exceptional curve occurs in
exactly one fibre, summing the numbers of components other than the
central components gives \eqref{eq:n_product-quotient}.

Moreover, by \cite{Ser96}*{Proposition 2.2} we get
\begin{proposition} \label{prop: irregularity of P.Q.}
The irregularity of a product-quotient surface $S$ is 
\begin{equation}
q(S)=g(C_1/G)+g(C_2/G).
\end{equation}
\end{proposition}

\subsection{Surfaces isogenous to a product of unmixed type} We first consider the case where $g(C_1), \, g(C_2)\geq 2$ and the action
of $G$ on $C_1\times C_2$ is free. In this situation, $X$ is smooth and
\begin{equation}
S=X=(C_1\times C_2)/G
\end{equation}
is called a \emph{surface isogenous to a product of unmixed type}. 

Because of the assumption on $g(C_i)$, the canonical divisor $K_{C_1\times C_2}$ is ample. Since the quotient map
$C_1\times C_2\longrightarrow S$ is finite and étale, the canonical divisor $K_S$ is ample as well.
Consequently, $S$ is a minimal surface of general type.

\begin{proposition} \label{prop:cone_isogenous_non_polyhedral}
Let $S = (C_1 \times C_2)/G$ be  a surface isogenous to a product of unmixed type.  If $\rho(S) \geq 3$, then  $\overline{\operatorname{Eff}}(S)$ is non-polyhedral. In particular,  $S$  is not a Mori dream space. 
\end{proposition}
\begin{proof}
Since the action of $G$ on $C_1 \times C_2$ is free, by \cite{Ser96}*{Theorem 2.1} 
all the singular fibres of $f_i$  are multiples  of smooth curves, hence $n(f_i) =0$ and the claim follows from Corollary \ref{cor:main_corollary}. 
\end{proof}

Recall that a \emph{non-trivial correspondence} between $C_1$ and $C_2$
is a possibly reducible curve $\Gamma\subset C_1\times C_2$ such that
both natural projections
\begin{equation}
\Gamma\longrightarrow C_1,
\quad
\Gamma\longrightarrow C_2
\end{equation}
are surjective.

\begin{proposition} \label{prop:cone_isogenous_non_polyhedral_correspondence}
Let $S=(C_1\times C_2)/G$ be a surface isogenous to a product of
unmixed type. Assume that there exists a non-trivial $G$-invariant
correspondence $\Gamma\subset C_1\times C_2$ such that
$\Gamma^2\leq0$. Then $\overline{\operatorname{Eff}}(S)$ is
non-polyhedral.
\end{proposition}
\begin{proof}
Let 
\begin{equation}
\pi \colon C_1 \times C_2 \to  S = (C_1 \times C_2)/G
\end{equation}
 be the quotient map   and set $B=\pi(\Gamma)$. Since $\Gamma$ is $G$-invariant and $\pi$ is unramified, we get $\pi^*B=\Gamma$. We claim that the three classes  
\begin{equation}
F_1, \; \; F_2, \;\; B
\end{equation}
 are linearly independent in $N^1(S)$. In fact, suppose that there exists a linear relation
\begin{equation}
\lambda_1 F_1 + \lambda_2 F_2 =\mu B,
\end{equation} 
  where $\lambda_1, \, \lambda_2, \, \mu \in \mathbb{R}$ and we can assume $\mu \geq 0$. Pulling back to $C_1 \times C_2$ via $\pi $, we get
\begin{equation}
\tau_1 C_2 + \tau_2 C_1 = \mu \Gamma,
\end{equation}  
 where $\tau_1 = \lambda_1 \cdot |G|$, $\tau_2 = \lambda_2 \cdot |G|$. Intersecting with $C_1$ we get $\tau_1 \geq 0$, and intersecting with $C_2$ we get $\tau_2 \geq 0$. Thus, intersecting with $\Gamma$ we get
\begin{equation}
 0 \leq (\tau_1 C_2 + \tau_2 C_1) \Gamma = \mu \Gamma^2 \le 0.
 \end{equation} 
 Then $(\tau_1 C_2 + \tau_2 C_1) \Gamma=0$. As $\Gamma$ is a non-trivial correspondence,  both products $C_1 \Gamma $ and $C_2 \Gamma$ are strictly positive. Then $\tau_1=\tau_2=\mu=0$ and so $\lambda_1=\lambda_2=0$. 
 
 Thus $F_1$, $F_2$, and $B$ are linearly independent, hence
$\rho(S)\geq3$. The conclusion now follows from Proposition
\ref{prop:cone_isogenous_non_polyhedral}.
  \end{proof}

\begin{remark}\label{rmk:invariance_essential}
In Proposition
\ref{prop:cone_isogenous_non_polyhedral_correspondence}, the condition
that $\Gamma$ be $G$-invariant is essential. Indeed, if
$C_1=C_2=C$, then the diagonal $\Delta\subset C\times C$ is a non-trivial correspondence 
having  self-intersection  $\Delta^2=2-2g(C)$,   which is negative whenever $g(C)\geq2$.

However, the actions of $G$ on the two copies of $C$ need not coincide,
and in this case $\Delta$ need not be $G$-invariant. There are examples
of free diagonal $G$-actions for which this happens and the resulting
surface $S=(C\times C)/G$ is a Mori dream space. One such example is the Beauville surface
described in Remark \ref{rmk:7_is_necessary}.
\end{remark}

We now give some examples of surfaces to which Proposition \ref{prop:cone_isogenous_non_polyhedral_correspondence} may be applied. The first example is due to J. Starr, see \cite{MO2026}.

\begin{exampleJS} \label{ex:JS}
Let $p\geq 7$ be a prime number and consider the finite group $G=\mathbb{Z}_p \times \mathbb{Z}_p$. The three elements
\begin{equation}
g_1=(1, \, 0), \quad g_2=(0, \, -1), \quad g_3=(-1, \, 1)
\end{equation}
generate $G$ and satisfy $g_1+g_2+g_3=0$, hence there is an associated $G$-cover
\begin{equation}
\pi \colon C \to \mathbb{P}^1,
\end{equation}
branched over $0, \, 1, \, \infty$ such that 
\begin{itemize}
\item the stabilizer of every point lying over $0$ is $\langle g_1 \rangle$;
\item the stabilizer of every point lying over $1$ is $\langle g_2 \rangle$;
\item the stabilizer of every point lying over $\infty$ is $\langle g_3 \rangle$.
\end{itemize}
If $x \in C$ and $g \in G$, we denote the action of $g$ on $x$ by $g\cdot x$, keeping the additive notation so that 
\begin{equation}
g \cdot (h \cdot x) = (g + h) \cdot x.
\end{equation}

By the Riemann-Hurwitz formula, $g(C)= \frac{1}{2} \, (p-1)(p-2)$. Now,  let us write $\mathbb{F}_p$ for the field with $p$ elements, whose additive group is $\mathbb{Z}_p$, and consider the automorphism 
$\phi \colon G \to G$ given by the matrix
\begin{equation}
\begin{pmatrix}
2 & -3 \\
1 & -2
\end{pmatrix} \in \operatorname{GL}_2(\mathbb{F}_p).
\end{equation}
We have
\begin{equation}
\phi(g_1)=(2, \, 1), \quad \phi(g_2)=(3, \, 2), \quad \phi(g_3)=(-5, \, -3)
\end{equation}
and so, since $p \ge 7$,
\begin{equation} \label{eq:intersection_stabilizers}
\langle g_i \rangle \cap \langle \phi(g_j) \rangle = \{(0,0)\} \quad \textrm{for all } i, \,j.  
\end{equation}
Therefore the action of $G$ on $C \times C$ defined by
\begin{equation}
g \cdot (x, \, y) := (g \cdot x, \; \phi(g) \cdot y)
\end{equation}
is free, and we obtain a surface isogenous to a product of unmixed type
\begin{equation}
S = (C \times C)/G,
\end{equation}
which is a minimal surface of general type and satisfies $q(S)=0$ by Proposition \ref{prop: irregularity of P.Q.}.

Moreover,  as will be relevant in the sequel,  $1$ is an eigenvalue of $\phi$, and we have
\begin{equation}
\ker(\phi - \operatorname{id}) = \langle (3,\, 1) \rangle, \quad \operatorname{im}(\phi - \operatorname{id}) = \langle (1,\, 1) \rangle.
\end{equation}
For any $g \in G$, we can consider the graph of the automorphism $x\mapsto g\cdot x$, i.e. 
\begin{equation} \label{eq:Gamma_g}
\Gamma_g = \{ (x, \; g\cdot x) \; \; | \; \; x \in C \} \subset C \times C.
\end{equation}
It is a smooth divisor isomorphic to $C$ and such that $\Gamma_g^2=2-2g(C)$. Note that $\Gamma_{(0, \, 0)}$ is the diagonal $\Delta \subset C \times C$, and that 
\begin{equation} \label{eq:Gamma_intersections}
\Gamma_{h_1} \cap \Gamma_{h_2} \neq \emptyset \quad \textrm{if and only if} \quad h_1 - h_2 \in \langle g_1 \rangle \cup \langle g_2 \rangle \cup \langle g_3 \rangle. 
\end{equation}
Moreover, for all $g \in G$, we have
\begin{equation}
\begin{split}
g \cdot \Delta & = \{(g \cdot x, \; \phi(g) \cdot x) \; \; | \; \; x \in C \}  \\
 & = \{ (y, \; (\phi(g) - g) \cdot y)  \; \; | \; \; y \in C \} = \Gamma_{\phi(g)-g}.
\end{split}
\end{equation}

\begin{theorem} \label{thm:surface_JS_non_mori_dream}
For all prime numbers $p \geq 7$, the pseudo-effective cone $\overline{\operatorname{Eff}}(S)$ of the surface $S$ constructed above is non-polyhedral.  In particular, $S$ is not a Mori dream space.
\end{theorem}
\begin{proof}
Let $\Gamma$ be the reduced $G$-orbit of the diagonal. Since
$\ker(\phi-\operatorname{id})$ is precisely the stabilizer of $\Delta$,
we have
\begin{equation}
\Gamma =
\sum_{\bar{g} \in G/\ker(\phi-\operatorname{id})}g\cdot\Delta
=
\sum_{h\in\operatorname{im}(\phi-\operatorname{id})}\Gamma_h.
\end{equation}
In particular, $\Gamma$ is a $G$-invariant non-trivial correspondence.

If  $h_1, \, h_2 \in \operatorname{im}(\phi - \operatorname{id})$, then also $h_1 - h_2\in \operatorname{im}(\phi - \operatorname{id})$.  Since $\operatorname{im}(\phi - \operatorname{id}) $ intersects each stabilizer $\langle g_i \rangle$ only at the identity,  by \eqref{eq:Gamma_intersections} we get 
\begin{equation}
\Gamma_{h_1} \Gamma_{h_2}=0 \quad \text{whenever} \quad h_1 \neq h_2. 
\end{equation}
Since $\operatorname{im}(\phi-\operatorname{id})$ has order $p$, we
obtain
\begin{equation}
\Gamma^2= \left(\sum_{h \in \operatorname{im}(\phi - \operatorname{id})} \Gamma_{h} \right)^2 = \sum_{h \in \operatorname{im}(\phi - \operatorname{id})} \Gamma_h^2 = p(2-2g(C))<0,
\end{equation}
and the claim follows from Proposition \ref{prop:cone_isogenous_non_polyhedral_correspondence}.
\end{proof}
\end{exampleJS}

\begin{remark}\label{rmk:7_is_necessary}
For primes $p\leq5$, it is not possible to find a matrix
$A\in\operatorname{GL}_2(\mathbb{F}_p)$ satisfying all the conditions
imposed in the example above. Indeed, set
\begin{equation}
\Sigma:=
\bigl\{
\langle g_1\rangle, \, \langle g_2 \rangle, \, \langle g_3\rangle
\bigr\}
\subset\mathbb{P}^1(\mathbb{F}_p).
\end{equation}
The diagonal action of $G$ on $C \times C$ is free precisely when
\begin{equation}
A(\Sigma)\cap\Sigma=\varnothing.
\end{equation}
For $p\leq3$, this is impossible, as we can see by counting the elements of
$\mathbb P^1(\mathbb F_p)$.

For $p=5$, a free diagonal action does exist and yields the classical
surface with
\begin{equation}
p_g=q=0,\quad K^2=8,
\end{equation}
described by Beauville in \cite{Beau96}*{Ex. X.13(4)}. In this case,
however,
\begin{equation}
A(\Sigma)=\mathbb{P}^1(\mathbb{F}_5)\setminus\Sigma,
\end{equation}
so $A$ fixes no line in $\mathbb{F}_5^2$.\footnote{Here the six points of \(\mathbb P^1(\mathbb F_5)\) parametrize the six subgroups of order $5$ of  $G=\mathbb{Z}_5 \times \mathbb{Z}_5$. The three stabilizer subgroups $\langle g_i \rangle$ form one half, and their images under  $A$  form the other half. Since  $A$ exchanges the two halves, it cannot admit an eigenline.} In particular, $1$ is not an
eigenvalue of $A$, and the construction used in the preceding example
does not apply. Indeed, Beauville's surface is a Mori dream space, see
\cite{KL19}*{Proposition 3.4}.
\end{remark}

More examples of surfaces to which Proposition \ref{prop:cone_isogenous_non_polyhedral_correspondence} may be applied can be constructed as follows.

\begin{proposition} \label{prop:cone_isogenous_non_polyhedral_via_factorization}
Let $S=(C_1 \times C_2)/G$ be a surface isogenous to a product of unmixed type, let $H$ be a normal subgroup of $G$ and set $Q:=G/H$.

Assume that 
\begin{equation}
\eta_i \colon C_i/H \to C_i/G, \quad i=1, \, 2 
\end{equation}
are isomorphic as $Q$-covers, and that
\begin{equation}
g(C_1/H)=g(C_2/H) \geq 1.
\end{equation}
Then   $\overline{\operatorname{Eff}}(S)$ is non-polyhedral, and so  $S$ is not a Mori dream space.
\end{proposition}

\begin{proof}
After choosing an isomorphism of the two $Q$-covers, we identify 
\begin{equation}
B:=C_1/G=C_2/G,  \quad  C:=C_1/H=C_2/H, \quad  \eta := \eta_1 = \eta_2.
\end{equation}
By assumption there are two factorizations
\begin{equation} \label{eq:factorization_C_i}
\begin{tikzcd}[row sep=3em, column sep=1em]
    C_1 \arrow[rr, "\rho_1"] \arrow[swap, dr, "\varphi_1"] & & C \arrow[dl, "\eta"] \\
    & B \\[-4.6em]
 \end{tikzcd}
\qquad
\begin{tikzcd}[row sep=3em, column sep=1em]
    C_2 \arrow[rr, "\rho_2"] \arrow[swap, dr, "\varphi_2"] & & C \arrow[dl, "\eta"] \\
    & B \\[-4.6em]
 \end{tikzcd}
\end{equation}
\medskip
of the $G$-covers $\varphi_i \colon C_i \to B$. 

The $(H \times H)$-cover 
\begin{equation}
\rho = \rho_1 \times \rho_2 \colon C_1 \times C_2 \to C \times C
\end{equation}
is equivariant with respect to the quotient homomorphism
$G\to Q$: explicitly, if $\bar g$ denotes the image of $g\in G$ in
$Q$, then
\begin{equation}
\rho (g\cdot(x, \, y)) = \bar g\cdot\rho(x, \,y).
\end{equation}
Thus there is a commutative diagram  
\begin{equation}
\begin{tikzcd}
C_1 \times C_2 \arrow[r, "\rho"] \arrow[swap, d, "\pi"] & C \times C\arrow[d, "\bar{\pi}"]\\
S=(C_1 \times C_2)/G \arrow[swap, r, "\bar{\rho}"] & (C \times C)/Q \mathrlap{,}
\end{tikzcd}
\end{equation}
 where $Q$ acts diagonally on $C \times C$. 

Since the action of $Q$ on the two factors of $C \times C$ is the same, the diagonal $\Delta \subset C \times C$ is $Q$-invariant. We set 
\begin{equation}
\Gamma := \rho^*(\Delta) \subset C_1 \times C_2 , \quad R = \bar{\pi}(\Delta) \subset (C \times C)/Q,
\end{equation}
and we note that $\Gamma$ is a non-trivial correspondence in $C_1 \times C_2$. 

Since both $\bar{\pi}$ and $\bar{\rho}$ are finite morphisms between
normal varieties, the pull-back of Weil divisors is well-defined; see
for instance \cite[Lemma 1.5]{DasOu24}.

Moreover, the action of $Q$ on $C$ is faithful, so the same is true for its restriction
to $\Delta\simeq C$. Therefore $\bar{\pi}$ is unramified at the generic
point of $\Delta$. Since $\Delta$ is the unique prime divisor lying over
$R$, it follows that
\begin{equation}
\bar{\pi}^*R=\Delta.
\end{equation}
Summing up, we get
\begin{equation} \label{eq:Gamma_G_invariant}
\Gamma = \rho^* (\Delta) = \rho^*(\bar{\pi}^*R) = \pi^*(\bar{\rho}^*R),
\end{equation}
that is, the correspondence $\Gamma$ is $G$-invariant. Finally, since $g(C) \geq 1$ we have $\Delta^2 \le 0$ and consequently $\Gamma^2 \le 0$, so we can conclude using Proposition \ref{prop:cone_isogenous_non_polyhedral_correspondence}.   
\end{proof}

We use Proposition \ref{prop:cone_isogenous_non_polyhedral_via_factorization} to construct  other examples of minimal surfaces of general type with $q=0$ that are not Mori dream spaces.  The case $p=3$ was explained to us by W. Sawin in \cite{MO2026}.

\begin{exampleWS}
Let $p \geq 3$ be a prime number and take three homogeneous polynomials $f(x_0,\, x_1), \, g(x_0,\,x_1), \, h(x_0, \, x_1)$ of degree $p$ such that the product $fgh$ has $3p$ distinct roots in ${\mathbb P}^1$.

Take ${\mathbb P}^3$  with homogeneous coordinates $x_0,\,x_1,\,y_0,\,y_1$ and let $C_1 \subset {\mathbb P}^3$ be the curve  
\begin{equation} \label{eq:WS_C1}
\begin{cases}
y_0^p=f & \\
y_1^p=g & \\
\end{cases}
\end{equation}
The curve $C_1$ is smooth by the assumption on $fgh$. We consider the faithful action of $ G:=\mathbb{Z}_p \times \mathbb{Z}_p $ on $C_1$ induced by the following action on ${\mathbb P}^3$:
\begin{equation}\label{eq:action of G 1}
(a,b)  \cdot (x_0,\, x_1, \, y_0, \, y_1)=(x_0, \, x_1,\, \xi^a y_0, \, \xi^b y_1)
\end{equation}  
where $\xi = e^{2 \pi i /p}$. The induced $G$-cover
$\varphi_1 \colon C_1 \to C_1/G \simeq \mathbb{P}^1$ is given by $(x_0, \, x_1, \, y_0, \, y_1) \mapsto (x_0,\,x_1)$.

Now take the weighted projective space ${\mathbb P}(1,\,1,\,1,\,2)$ with weighted coordinates $x_0,\,x_1,\,z_0,\,z_1$; here $z_1$ is the variable of weight $2$. Let $C_2 \subset {\mathbb P}(1,\,1,\,1,\,2)$  be the curve
\begin{equation} \label{eq:WS_C2}
\begin{cases}
z_0^p=h & \\
z_1^p=fg & \\
\end{cases}
\end{equation}
The curve $C_2$ is smooth by the assumption on $fgh$.
We consider the action of $G $ on $C_2$ induced by the following action on ${\mathbb P}(1,\,1,\,1,\,2)$:  
\begin{equation}\label{eq:action of G 2}
(a,\,b)  \cdot (x_0, \,x_1,\,z_0,\,z_1)=(x_0, \,x_1,\, \xi^{a-b} z_0,\, \xi^{a+b} z_1)
\end{equation}  
Since $p \ge 3$, the action \eqref{eq:action of G 2} is faithful on $C_2$, so we obtain a further $G$-cover
$\varphi_2 \colon C_2\to C_2/G \simeq \mathbb{P}^1$
given by  $(x_0, \, x_1, \, z_0, \, z_1) \mapsto (x_0, \, x_1)$.\footnote{If $p=2$ then the element $(1,\,1) \in G$ would act trivially on $C_2$.}

Let  $H$ be  the cyclic subgroup of $G$ generated by $(1, \, -1)$. 

The $H$-invariant subring of ${\mathbb C}[x_0, \, x_1, \, y_0, \, y_1]/(y_0^p-f, \, y_1^p-g)$, with respect to the action  induced by  \eqref{eq:action of G 1}, is generated by $x_0$, $x_1$, and $y_0y_1$.  Setting $y:=y_0y_1$ we can write $C_1/H$ as the smooth curve $\{y^p=fg \} \subset {\mathbb P}(1,\, 1,\,2)$ and the cover $\eta_1 \colon C_1/H \to {\mathbb P}^1$ is given by $(x_0, \, x_1, \, y) \mapsto (x_0,\,x_1)$.

The $H$-invariant subring of ${\mathbb C}[x_0, \, x_1, \, z_0, \,z_1]/(z_0^p-h,\, z_1^p-fg)$, with respect to the action  induced by \eqref{eq:action of G 2}, is generated by $x_0$, $x_1$ and $z_1$. We can write $C_2/H$ as the smooth curve $\{z_1^p=fg \} \subset {\mathbb P}(1,1,2)$ and the cover $\eta_2 \colon C_2/H \to {\mathbb P}^1$ is given by $(x_0, \, x_1, \, z_1) \mapsto (x_0,\,x_1)$. 

The homomorphism
\begin{equation}
G\longrightarrow \mathbb Z_p,\qquad (a,b)\longmapsto a+b
\end{equation}
has kernel $H$, and hence identifies $Q=G/H$ with $\mathbb Z_p$. Under
this identification, the induced actions on $C_1/H$ and $C_2/H$ are
given respectively by
\begin{equation}
y\longmapsto \xi^{a+b}y,
\qquad
z_1\longmapsto \xi^{a+b}z_1.
\end{equation}
Consequently, the identification $y=z_1$ is $Q$-equivariant, and
therefore $\eta_1$ and $\eta_2$ are isomorphic as $Q$-covers.

\begin{theorem} \label{thm:surface_WS_non_mori_dream}
For all prime numbers $p \geq 3$, consider the actions of $G$ on $C_1$ and $C_2$ induced by  \eqref{eq:action of G 1} and \eqref{eq:action of G 2}, respectively.  Let $S:=(C_1 \times C_2)/G$, where $G$ acts diagonally. Then $S$ is a minimal surface of general type isogenous to a product of unmixed type, with $q(S)=0$, such that $\overline{\operatorname{Eff}}(S)$ is non-polyhedral. In particular, $S$ is not a Mori dream space.
\end{theorem}
\begin{proof}
 Since $x_0$ and $x_1$ do not vanish simultaneously on $C_1$, a standard computation using \eqref{eq:action of G 1} shows  that the elements of $G$ fixing some point of $C_1$ are exactly those of the form $(a, \, 0)$ or $(0, \, b)$. 
 
 By \eqref{eq:action of G 2}, since $x_0$ and $x_1$ do not vanish simultaneously on $C_2$, the only element of this form fixing a point of $C_2$ is $(0, \, 0)$. Thus the diagonal action of $G$ on $C_1 \times C_2$ is free, and $S$ is isogenous to a product of unmixed type. Moreover,  $q(S)=0$ by Proposition \ref{prop: irregularity of P.Q.}.

Furthermore, we showed above that  $\eta_i \colon C_i/H \to  \mathbb{P}^1$ are isomorphic as $G/H$-covers. 
They have degree $p$ and are totally ramified over the $2p$ distinct zeros of $fg$. Thus the Riemann-Hurwitz formula gives
\begin{equation*}
2g(C_i/H)-2 = -2p + 2p(p-1),
\end{equation*}
hence $g(C_i/H)=(p-1)^2$. This also shows that $g(C_i) \geq 2$, hence $S=(C_1 \times C_2)/G$ is a minimal surface of general type.  

Finally, the claim about $\psEff(S)$ follows from Proposition \ref{prop:cone_isogenous_non_polyhedral_via_factorization}. 
\end{proof}
\end{exampleWS}

\bigskip

\subsection{Product-quotient surfaces with geometric genus zero}
 In \cite{BaPi16}*{Definitions 2.1-2.3} the authors associate to every surface $X$ whose singular points are cyclic quotient singularities a rational number $\gamma(X)$ that depends only on the singularities of $X$. Then they prove the following result, see \cite{BaPi16}*{Prop. 4.2}:
 \begin{proposition}
 If $X = (C_1 \times C_2)/G$ is the quotient model of a product-quotient surface $S$, then
 $\gamma(X) \in \mathbb{Z}$ and  $p_g(S)+\gamma(X)\ge0$. Furthermore, denoting by $f_i \colon S \to C_i/G$ the two natural fibrations, we have
\begin{equation}
	h^{1,1}(S)-n(f_i)-2=2(p_g(S)+\gamma(X)).
\end{equation}
\end{proposition}

Thus, in this case, Theorem \ref{thm: main} has the following reformulation.
\begin{corollary}
Let $S$ be a product-quotient surface with quotient model $X$.
If  $\overline{\operatorname{Eff}}(S)$ is polyhedral, then  
\begin{equation} 
h^{1,1}(S)-\rho(S) = 2(p_g(S)+\gamma(X)).
\end{equation}
\end{corollary}

If moreover $p_g(S)=0$, by the Lefschetz $(1,\, 1)$-theorem we have $\rho(S)=h^{1, 1}(S)$. Thus by Corollary \ref{cor:main_corollary} we obtain

\begin{theorem}\label{thm:not_Mori_Dream_PQ}
Let $S$ be a product-quotient surface with $p_g(S)=0$.  If $\gamma(X)>0$ then   $\overline{\operatorname{Eff}}(S)$ is non-polyhedral, hence $S$ is not a Mori dream space.
\end{theorem}

In \cite{BaPi16}*{Tables 2–5} one can find more than 100 families of surfaces with $p_g=q=0$ and $\gamma>0$. By Theorem \ref{thm:not_Mori_Dream_PQ}, these surfaces are not Mori dream spaces. Moreover, none of them is of general type \cite{BaPi16}*{Subsection 7.3}. 

To the best of our knowledge, only four examples of product-quotient surfaces of general type with $p_g=q=0$ and $\gamma >0$ are known to exist;  they are described  in Table \ref{table:ex_pg=q=0_gen_type_not_min}.  By Theorem \ref{thm:not_Mori_Dream_PQ}, they are not Mori dream spaces.

Note that these examples are necessarily non-minimal (Remark  \ref{rmk:non_minimal_gamma_positive}). It is an interesting open problem to determine whether their minimal models  are Mori dream spaces.

 The last two surfaces, namely those with $\gamma(X)=2$, contain precisely
two $(-1)$-curves $E_1$ and $E_2$, where $E_i$ is contracted by
$f_i\colon S\to C_i/G$. For each $i=1, \, 2$, contracting $E_i$ produces a surface $S'$ with 
\begin{equation}
p_g(S')=q(S')=0,\quad K_{S'}^2=0,
\end{equation}
which still satisfies the hypothesis of
Corollary \ref{cor:main_corollary}.\footnote{Let $f\colon S\to C$ be a fibration and let $E\subset S$ be a
$(-1)$-curve contracted by $f$. Then $f=f' \circ c$, where $c \colon S \lr S'$ is the contraction of $E$ and $f'$ is a fibration on $S'$. 
As $\rho(S)=\rho(S')+1$ and $n(f)=n(f')+1$, we see that $\rho(S) > n(f)+2$ if and only if $\rho(S') > n(f')+2$.} Thus $S'$ is not a Mori dream space either.

\begin{small}
	\begin{table*}[h!]
	
	\centering
	
	\renewcommand{\arraystretch}{1.2}
	\begin{tabular}{c | c| c| c|c| c|  c|c }
		$\gamma $ & $K_S^2$ & $K_{\widetilde{S}}^2$ & $G$ & $\operatorname{Sing}(X)$ &	 $t_1$ & $t_2$ &  Reference\\
		\hline
$1$ & 	$ 1 $ & $3$& $\PSL_2(\mathbb{F}_7)$   &$1/7,\, 2/7,\, 4/7$ & $ 3,\, 3,\, 7 $ & $ 2,\, 4,\, 7 $ & \cite{BaPi12}*{Section 5} \\
$1$ & 	$ 0 $ & $1$& $\GL_2(\mathbb{F}_4)$  & $2 \times 1/6,\, 2 \times 1/2,\, 2/3$ & $ 4,\, 4,\, 6 $ & $ 2,\, 4,\, 6$ &\cite{BaPi16}*{Section 7.1}\\
$2$ & 	$ -1 $ & $1$& $\mathbb Z_5 \times \mathbb Z_5$ & $5 \times 1/5$ & $ 5,\, 5,\, 5 $ & $ 5,\, 5,\, 5 $ & \cite{BaPi16}*{Section 7.2}\\
$2$ & 	$ -1 $ & $1$& $\mathbb Z_5 \times \mathbb Z_5$ & $5 \times 1/5$ & $ 5,\, 5,\, 5 $ & $ 5,\, 5,\, 5 $ & \cite{Fede25}*{Theorem 5.5} \\

	\end{tabular}
	\bigskip	
	
	\caption{Known examples of product-quotient surfaces of general type with $p_g=q=0$ and $\gamma >0$. Here $\widetilde{S}$ is the minimal model of $S$, whereas $t_i$ denotes the signature of the $G$-cover
$\varphi_i\colon C_i\to C_i/G\simeq\mathbb{P}^1$. The fraction $\frac{a}{n}$ indicates a cyclic quotient singularity of type $\frac{1}{n}(1, \, a)$.}
	\label{table:ex_pg=q=0_gen_type_not_min}
\end{table*}
\end{small}


\begin{remark} \label{rmk:non_minimal_gamma_positive}
A product-quotient surface $S$ of general type with $p_g(S)=q(S)=0$ is minimal if and only if $\gamma(X)=0$, see \cite{BaPi16}*{Theorem 6.2}. These surfaces are listed in \cite{BaPi12}*{Tables 1 and 2}, and it is an interesting problem to classify those which are Mori dream spaces.
\end{remark}

\section{The Craighero-Gattazzo surface}

The Craighero-Gattazzo surface $V$ is a particular numerical Godeaux
surface, namely, a minimal surface of general type with
\begin{equation}
p_g(V)=q(V)=0, \quad K_V^2=1.
\end{equation}
It was first constructed by Craighero and Gattazzo \cite{CG}. Later, it
was studied by Dolgachev and Werner \cite{DW}, who proved, among other
things, that $H_1(V,\mathbb Z)=0$ and claimed that $V$ is simply
connected. The proof of this last result contained a gap
\cite{DWerr}, and the fact that $V$ is simply connected was proved only
subsequently by Rana, Tevelev and Urz\'ua \cite{RTU}, using different
methods.

The aim of this section is to show that the Craighero-Gattazzo surface
is not a Mori dream space by applying Corollary
\ref{cor:main_corollary} to a special genus $2$ pencil on $V$, studied
in \cite{DW}. We will use the notation of that paper throughout.

The Craighero--Gattazzo surface $V$ is defined as the minimal resolution
of singularities of a quintic surface $S\subset\mathbb P^3$ with four
elliptic singularities; see \cite{DW}*{Section 2}. The quintic $S$ is invariant with respect to the order $4$ automorphism of $\mathbb{P}^3$
\begin{equation} \label{eq:sigma}
\sigma \colon [x:y:z:t] \mapsto [t:x:y:z]
\end{equation}
which has the four fixed points
\begin{equation}
\begin{split}
P_0& =[1:1:1:1], \quad P_1=[1:i:-1:-i], \\ 
P_2&=[1:-i:-1:i],\quad Q_0=[1:-1:1:-1]. 
\end{split}
\end{equation}
The set of fixed points of
\begin{equation} \label{eq:sigma^2}
\sigma^2 \colon [x:y:z:t] \mapsto [z:t:x:y]
\end{equation} 
is equal to the union of the two lines
\begin{equation}
\begin{split}
r & = \{x+z=y+t=0 \} = \langle P_1, \, P_2 \rangle \\
r' & = \{x-z=y-t=0 \} = \langle P_0, \, Q_0 \rangle,
\end{split}
\end{equation}
and the elliptic singularities of $S$ are at the four reference points 
\begin{equation}
a_1=[1:0:0:0], \quad a_2=[0:1:0:0], \quad a_3=[0:0:1:0], \quad a_4=[0:0:0:1].
\end{equation}
Let
\begin{equation}
\pi\colon V\longrightarrow S
\end{equation}
be the minimal resolution of singularities, and set
$E_i:=\pi^{-1}(a_i)$. The curves $E_i$ are elliptic curves satisfying
\begin{equation}
E_i^2=-1,\qquad K_V\cdot E_i=1,
\end{equation}
where the second equality follows from adjunction. 

Let $H$ denote a hyperplane section of $S$. Since $S$ is a quintic
hypersurface, adjunction gives $K_S =  H$. Moreover, each $a_i$ is a
 is a simple
elliptic singularity of type $\widetilde{E}_8$; see
\cite[Theorem 7.6.4]{Is14}. Since $S$ is Gorenstein, we can write
\begin{equation}
K_V= \pi^*K_S+\sum_{i=1}^4\lambda_iE_i.
\end{equation}
Intersecting with $E_j$ and using
$\pi^*K_S\cdot E_j=0$, $K_V\cdot E_j=1$, and $E_j^2=-1$, we obtain
$\lambda_j=-1$. Therefore
\begin{equation}\label{eq:canonical_class_CG}
K_V = \pi^*H-E_1-E_2-E_3-E_4.
\end{equation}

The reader can find in \cite{DW}*{Section 2} a simple proof that $V$ is
of general type and satisfies $p_g(V)=q(V)=0, \, K_V^2=1$. Moreover, we will later need the fact that $V$ contains no
$(-2)$-curves; see \cite{DW}*{Theorem 6.2} and \cite{DWerr}. Since $V$
is a minimal surface of general type, this is equivalent to saying that
$K_V$ is ample.

The automorphisms $\sigma$ and $\sigma^2$ of $S$ both lift to
automorphisms of $V$, which we denote by the same symbols. The
automorphism $\sigma\colon V\to V$ cyclically permutes the exceptional
elliptic curves $E_1,E_2,E_3,E_4$, hence
\begin{equation}\label{eq:action_of_sigma2_on_Ei}
\sigma^2(E_1)=E_3,\quad \sigma^2(E_2)=E_4.
\end{equation}

The line $r$ is contained in the quintic $S$ and is disjoint from its
singular locus. Therefore its inverse image
\begin{equation}
R:=\pi^{-1}(r)
\end{equation}
is a smooth rational curve. It is fixed pointwise by the involution
$\sigma^2\colon V\to V$ and satisfies
\begin{equation}
K_V\cdot R=1,\qquad R^2=-3.
\end{equation}

The line $r'$ intersects $S$ at five smooth points. Their inverse images
on $V$ are precisely the isolated fixed points of $\sigma^2$. We denote
them by
\begin{equation}
\Sigma:=\{Q_0, \, Q_1, \, Q_2, \, Q_3, \, Q_4\},
\end{equation}
using the same notation for $Q_0\in S$ and its inverse image on $V$.

Now we consider the complete linear system
\begin{equation}
|3K_V-R|
\end{equation}
on $V$. It is a pencil without fixed part \cite{DW}*{Proposition 3.2}. Since $(3K_V-R)^2=0$,
the pencil has no base points and therefore defines a fibration
\begin{equation}
f\colon V\to\mathbb P^1.
\end{equation}
By Bertini's theorem, its general fibre $F$ is smooth, and the adjunction
formula shows that it has genus $2$, since
\begin{equation}
F^2=0,\quad K_VF=(3K_V-R)K_V=2.
\end{equation}
We are interested in the reducible fibres of $f$. As $K_V$ is ample, a fibre $F$ of $f \colon V \to \mathbb{P}^1$ has at most $K_VF=2$ irreducible components. In particular, if $F$ is reducible, we have $F=A+A'$, with $K_VA=K_VA'=1$. 
Moreover, since $(A+A')A=(A+A')A'=0$, and $0 \le p_a(A) = \frac{1}{2}{(3-AA')}$, we get only two numerical possibilities:
\begin{itemize}
\item[$(a)$] $AA'= -A^2=-(A')^2=1;$
\item[$(b)$] $AA'= -A^2=-(A')^2=3.$
\end{itemize}

\begin{lemma} \label{lem:no_other_reducible_fibres}
Case $(b)$ does not occur.
\end{lemma}
\begin{proof}
We have seen that the fixed locus of the involution $\sigma^2 \colon V \to V$ consists of the disjoint union of the smooth rational curve $R$ and the five isolated points in $\Sigma$. Then, using the topological fixed point formula \cite{KwasikSchultz}*{Theorem 1} and recalling that $q(V)=0$ implies $H^1(V, \, \mathbb{Q})=H^3(V, \, \mathbb{Q})=0$, we get 
\begin{equation}
\begin{split}
2+5= \chi_{\operatorname{top}}(\operatorname{Fix}(\sigma^2)) & =\sum (-1)^q \operatorname{trace} \; ((\sigma^{2})^{*} \;| \; H^q(V, \, {\mathbb Q})) \\
& =2+\operatorname{trace} \; ((\sigma^{2})^{*} \;| \; H^2(V, \, {\mathbb Q})),
\end{split}
\end{equation}
that is, $\operatorname{trace} \; ((\sigma^{2})^{*} \;| \; H^2(V, \, {\mathbb Q})) = 5$. 

On the other hand, since $p_g(V)=q(V)=0$ and $\chi_{\operatorname{top}}(V)=11$, we infer $b_2(V)=9$. 

Thus, the invariant and anti-invariant subspaces $H^2(V, \, {\mathbb Q})^{+}$ and $H^2(V, \, {\mathbb Q})^{-}$ for the involution $(\sigma^2)^* \colon H^2(V, \, {\mathbb Q}) \to H^2(V, \, {\mathbb Q})$ have respective dimensions
\begin{equation}
\dim H^2(V, \, {\mathbb Q})^{+}=7, \quad \dim H^2(V, \, {\mathbb Q})^{-}=2.
\end{equation}
By \eqref{eq:action_of_sigma2_on_Ei}, a basis for $H^2(V, \, {\mathbb Q})^{-}$ consists of
\begin{equation}
\{E_1-E_3, \; E_2 - E_4\}.
\end{equation}
Suppose that there exists a reducible fibre $A+A'$ as in $(b)$. Then
\begin{itemize}
\item $A+A'$ does not contain any of the five isolated fixed points of $\sigma^2$, as they lie on fibres of type (a) as shown in \cite{DW}*{Lemma 3.5, Corollary 3.6};\footnote{As $K_V$ is ample, the cases (ii) and (iii) in \cite{DW}*{Lemma 3.5} do not occur; cf. \cite{DW}*{p. 751 and Remark 3.9}.} 
\item moreover, neither $A$ nor $A'$ equals $R$, since
$(3K_V-R)R=6$. Since $A^2= (A')^2=-3$ is odd, \cite{DW}*{Lemma 3.4} implies that neither $A$ nor $A'$ is
$\sigma^2$-invariant;
\item $A+A'$ is $\sigma^2$-invariant, because every member of $|3K_V-R|$ is $\sigma^2$-invariant; see \cite{DW}*{p. 742}.
\end{itemize}
Thus $A'=\sigma^2(A)$, that is $A-A' \in H^2(V,\, {\mathbb Q})^-$. Therefore there exist $\lambda, \, \mu \in \mathbb{Q}$ such that, in $H^2(V, \, \mathbb{Q})$:
\begin{equation}
A-A' = \lambda(E_1-E_3)+\mu(E_2-E_4).
\end{equation}
Since $\lambda=(A-A')E_3$ and $\mu=(A-A')E_4$, both $\lambda$ and $\mu$ are integers. Then 
\begin{equation}
2(\lambda^2 + \mu^2) = -(A-A')^2 = -(A^2 -2 AA'+(A')^2) = 12
\end{equation}
provides the desired contradiction, because the equation $\lambda^2 + \mu^2=6$ has no integral solutions.
\end{proof}

So, all reducible fibres of $f$ are of type $(a)$.  We are now in a position to prove the main result of this section.
\begin{theorem} \label{thm: Craighero-Gattazzo_non_Mds}
The pseudo-effective cone $\overline{\operatorname{Eff}}(V)$ of the Craighero-Gattazzo surface $V$ is  non-polyhedral. In particular, $V$ is not a Mori dream space.

\end{theorem}
\begin{proof}
All reducible fibres of $f$ are of type $(a)$, hence they are not
$2$-connected.

By the theory of genus-$2$ fibrations, the number of fibres of $f$ that
are not $2$-connected is at most
\begin{equation}
K_V^2-2\chi(\mathcal{O}_V)+6=5,
\end{equation}
see \cite{CP}*{Lemma 4.1(1) and Theorem 4.13}.
Indeed, the divisor $\tau$ appearing there is supported on the points
corresponding to the fibres that are not $2$-connected, and its degree is
$K_V^2-2\chi(\mathcal{O}_V)+6$.

 Since $V$ contains no $(-2)$-curves,
\cite{DW}*{Lemma 3.5 and Corollary 3.6} show that the five points of
$\Sigma$ lie on five distinct fibres, all of type $(a)$. Therefore $f$
has exactly five reducible fibres. Since each of them has two irreducible
components, we obtain $n(f)=5$.

On the other hand, since $p_g(V)=0$, the Lefschetz $(1, \, 1)$-theorem gives
\begin{equation}
\rho(V)=h^{1,1}(V)=b_2(V)=9.
\end{equation}
Thus
\begin{equation}
\rho(V)=9>7=n(f)+2,
\end{equation}
and the result follows from Corollary
\ref{cor:main_corollary}.
\end{proof}

\section*{Acknowledgements}
All authors were partially supported by INdAM-GNSAGA. The first author held a research grant from INdAM, Istituto Nazionale di Alta Matematica. The second author thanks C. Fontanari for first pointing out the paper \cite{KL19}  and for interesting conversations on the problems raised there. We thank J. Buczy\'nski, J. Hofscheier, J. Keum, G. Occhetta, A. Rabindranath, L. Sol\'a Conde and M. Sch\"utt for useful discussions. We also thank C. Casagrande,  J. Keum  and A. Laface for having pointed out \cite{Casagrande}*{Corollary 2.16},  \cite{Keum} and \cite{CLTU}*{Proposition 2.3 and Remark 2.4}, respectively. We are also indebted to W. Sawin and J. Starr, who answered our questions on MathOverflow \cite{MO2026}.

\section*{AI disclosure}
All the ideas and the overall direction of this work were developed by the
authors. During the preparation of the paper, the authors used ChatGPT
(OpenAI) for technical assistance in checking calculations,
identifying relevant references, and improving the exposition. All
AI-assisted material was subsequently reviewed and edited by the authors,
who independently verified every mathematical argument, calculation, and
reference. The authors take full responsibility for the mathematical
correctness and final text of the paper.

\addtocontents{toc}{\protect\setcounter{tocdepth}{1}}



\end{document}